\documentclass{amsproc}

\usepackage{amsmath,amssymb}
\usepackage{graphicx}
\usepackage{pinlabel}
\usepackage{tikz}
\usepackage[section]{placeins}
\usepackage[hidelinks]{hyperref}

\newcommand{\Flow}{\operatorname{Flow}}
\newcommand{\type}{\operatorname{type}}
\newcommand{\ustar}{\mathbin{\underline{\ast}}}
\newcommand{\ostar}{\mathbin{\overline{\ast}}}
\newcommand{\Z}{\mathbb Z}

\newtheorem{thm}{Theorem}[section]
\newtheorem{prop}[thm]{Proposition}
\newtheorem{lem}[thm]{Lemma}
\newtheorem{cor}[thm]{Corollary}
\theoremstyle{definition}
\newtheorem{defn}[thm]{Definition}
\newtheorem{exmp}[thm]{Example}
\theoremstyle{remark}
\newtheorem{rem}[thm]{Remark}

\numberwithin{equation}{section}

\begin{document}

\title{Bridge numbers of virtual graphs and handlebody-links}

\author{Nikolaos Chantis}
\address{901 Palm Walk, Tempe, AZ 85281}
\email{nchantis@asu.edu}

\author{Puttipong Pongtanapaisan}
\address{1050 N Mills Ave, Claremont, CA 91711}
\email{puttip@pitzer.edu}


\begin{abstract}
We study bridge indices of graphs, virtual graphs, and their
handlebody classes.  We construct virtual ravel bouquets
for which the difference between bridge index and weighted overpass
bridge index is unbounded.  The lower bound uses exponentially many
colorings for one fixed flow in a $\Z_2$-family of biquandles.
For each negative Euler characteristic we construct classical ravels
of one fixed trivalent graph with arbitrarily large bridge index.
We also obtain unbounded even bridge indices for ravels of each
fixed bouquet rank.  A separate classical bouquet family has an
unbounded difference between its spatial-graph bridge index and the
bridge index of its regular neighborhood.  Finally, edge connected
sums of Suzuki curves give exact constrained and unconstrained bridge
indices whose difference is unbounded.
\end{abstract}

\maketitle

\section{Introduction}

Virtual spatial graphs, introduced by Fleming and Mellor
\cite{flemingMellor}, simultaneously extend virtual knots and classical
spatial graphs.  Their diagrams are related by the classical graph
Reidemeister moves, the virtual moves, and the virtual vertex move.
As with virtual knots, there is more than one useful notion of bridge
complexity.  A spatial graph introduces further choices: one may
allow vertices on both sides of a bridge surface, or require all
vertices to lie on the same side. One may also minimize over isotopic embeddings of a graph or over spines of its regular neighborhood.

We call the resulting quantities the \emph{unconstrained bridge
index} $b(G)$ and the \emph{constrained bridge index} $b_c(G)$.  We
also compare $b(G)$ with the weighted overpass bridge index $b_o(G)$ and with
the virtual handlebody bridge index $b_{\mathrm{hb}}([G]_{\mathrm{hb}})$,
obtained by minimizing over all handlebody-equivalent spatial
spines.  In the classical category, this is written
$b_{\mathrm{hb}}(N(G))$ for the regular neighborhood of $G$.
The evident inequalities
\[
 b_o(G)\leq b(G)\leq b_c(G)
 \quad\text{and}\quad
 b_{\mathrm{hb}}([G]_{\mathrm{hb}})\leq b(G)
\]
give no indication of the possible size of their defects.  Throughout,
$b$ and $b_c$ count all intersections with the bridge surface.  We
first define the unweighted overpass index $\widehat b_o$ by counting
overpass components, and then define $b_o$ by summing their weights,
just as for tree components of a bridge splitting.  In particular,
an arc has weight two and a $d$-valent pod has weight $d$.
The two overpass indices are minimized separately; the weights are
not recovered by multiplying an unweighted minimum by a fixed
constant.

For classical spatial graphs the equality $b_o(G)=b(G)$ is
\cite[Corollary~3.3]{blackwell2026}, with both sides rescaled to
our convention.  The first construction therefore takes place in
the virtual category.  Here and in the main results, $b_o$ and $b$
use the same component weights.

\begin{thm}\label{thm:kishino-overpass}
For every $n\geq2$, there is a virtual ravel $V_n$ whose underlying
abstract graph is an $n$-petal bouquet and such that
\[
 b_o(V_n)=2n,\qquad
 b(V_n)\geq2n-2+2\left\lceil\frac{n+1}{2}\right\rceil.
\]
In particular,
\[
                 b(V_n)-b_o(V_n)\geq n-1\longrightarrow\infty.
\]
\end{thm}

The proof in Section~\ref{subsec:virtual-ravels} uses a trivalent
expansion with at least $2^{n+1}$ colorings for one fixed $\Z_2$-flow.
Keeping the contribution of the $2n$-valent bouquet vertex in the
component-weight formula converts this coloring growth into the
stated weighted bridge gap.  The ravel condition requires all
constituent knots and links to be trivial, while allowing nontrivial
proper graph subgraphs.

Our second construction concerns classical ravels.  We use the
compatible tetrahedral clasping drawn in
Figure~\ref{fig:tetrahedral-clasping}.  Each inserted clasp admits four
colorings with a prescribed common boundary color.  Placing the
clasps in disjoint balls makes these choices independent, while
supporting them on three distinct half-edges at one vertex preserves
the ravel property.  This gives non-Eulerian trivalent graphs at
each negative Euler characteristic without changing the abstract
graph during the iteration.
For $x\geq1$, let $\Lambda_x$ be the planar trivalent graph formed
by doubling alternate edges of a $2x$-cycle, with
$\Lambda_1=\theta_3$.  Thus $\Lambda_2$ is a square with two
opposite edges doubled and $\Lambda_3$ is a hexagon with three
alternate edges doubled.

\begin{thm}\label{thm:anyEuler}
\leavevmode
\begin{enumerate}
\item[(i)] For every $x,k\geq1$, there is a classical ravel
$G_{x,k}$ with underlying abstract graph $\Lambda_x$ such that
\[
 \chi(G_{x,k})=-x,\qquad b(G_{x,k})\geq x+2k+2.
\]
\item[(ii)] For each $r\geq2$, classical ravel embeddings of the
$r$-petal bouquet have arbitrarily large even bridge indices.
\item[(iii)] There are classical ravel $r$-petal bouquets $P_r$
with
\[
 b(P_r)=2r+2,\qquad b_{\mathrm{hb}}(N(P_r))\leq r+4,
\]
so $b(P_r)-b_{\mathrm{hb}}(N(P_r))\geq r-2\to\infty$.
\end{enumerate}
\end{thm}

The almost unknotted bouquet examples with bridge index $2r+2$
are retained in Proposition~\ref{prop:bouquet-family}.  The clasping
construction in Theorem~\ref{thm:anyEuler}(i)--(ii) also applies to
bouquets.  The ravel conclusion requires trivial constituent links;
it does not require deletion of every edge to make the whole graph
planar.

The final result measures the cost of requiring all vertices to lie
on the same side of a bridge sphere.
\begin{thm}\label{thm:main-gap}
For each $n\geq1$, the $n$-fold edge connected sum $G_n$ of
Suzuki's $\theta_4$-curve is a classical ravel satisfying
\[
 b(G_n)=4n+2,\qquad b_c(G_n)=6n+2.
\]
Consequently $b_c(G_n)-b(G_n)=2n\to\infty$.
\end{thm}

One motivation is the possibility that biquandle colorings detect
bridge complexity invisible to quandle colorings; compare
\cite{elhamdadi2025bridging,murao2018relationship}.  For this purpose
one must check both the crossing conventions and the vertex
relations, including colors on short semiarcs.  A second motivation
is that many invariants of trivalent graphs are invariants of their
regular neighborhoods.  Comparing $b(G)$ with
$b_{\mathrm{hb}}(N(G))$ measures how much bridge complexity can be
lost when changes of spine are allowed.

Geometric constraints also arise in models of entanglement in
lattice tubes.  Atapour and coauthors studied
linking probabilities for pairs of polygons spanning a tube
\cite{atapour2010linking}.  Bridge constraints in this setting
are developed further in \cite{blair2026entanglement}, where
some link types require nonminimal bridge conformations to fit
the prescribed tube.

Throughout the paper, graphs are finite and have no vertices of
valence one.  Vertices of valence two are suppressed unless they are
useful in describing a move.  These conventions include links by
regarding a link component as a graph with no vertices.

\section{Bridge splittings and component weights}\label{sec:bridge}

\begin{defn}
A \emph{virtual spatial graph diagram} is a generic immersion of a
finite graph in $\mathbb R^2$ whose double points are marked as
classical or virtual crossings.  A \emph{virtual spatial graph} is an
equivalence class of such diagrams under planar isotopy, the
classical graph Reidemeister moves, the virtual Reidemeister moves,
and the virtual vertex move shown in
Figure~\ref{fig:virtual-moves}; see \cite{flemingMellor,yokota1996}.

A \emph{virtual graph tangle diagram} is defined in the same way in a
closed half-plane, with some degree-one boundary vertices on its
boundary line.  Equivalence of tangles fixes those boundary points.
\end{defn}

\begin{figure}[htbp]
  \centering
  \includegraphics[width=.7\linewidth]{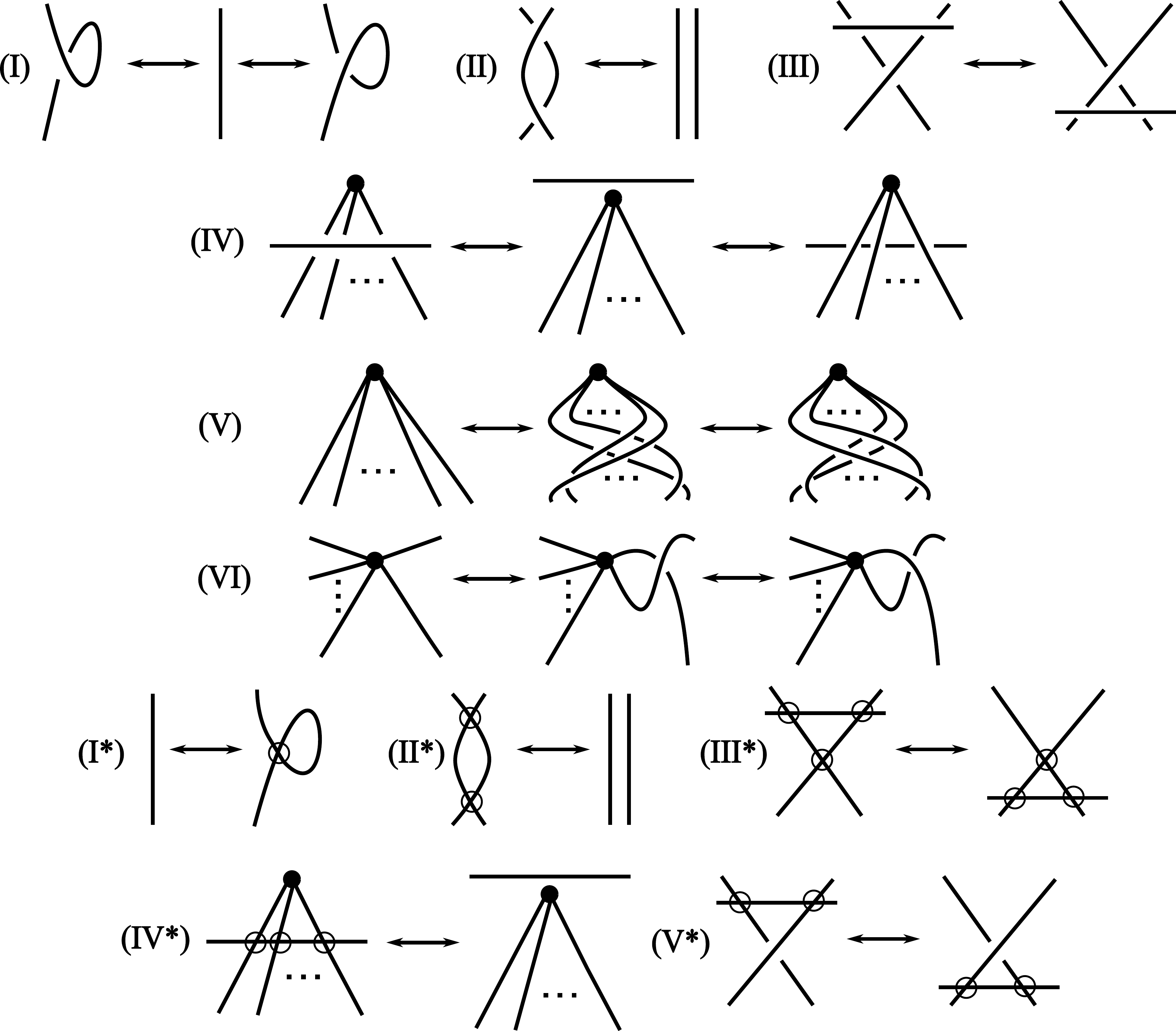}
  \caption{The moves for virtual spatial graph diagrams.}
  \label{fig:virtual-moves}
\end{figure}

A virtual braid may be stacked on a tangle at its boundary, as in
Figure~\ref{fig:stack}.  The reverse braid is obtained by reversing
the order of the crossings and changing every classical crossing to
its inverse.  A braid followed by its reverse cancels by classical
and virtual Reidemeister II moves; see Figure~\ref{fig:reverse-braid}.

\begin{figure}[htbp]
  \centering
  \includegraphics[width=.48\linewidth]{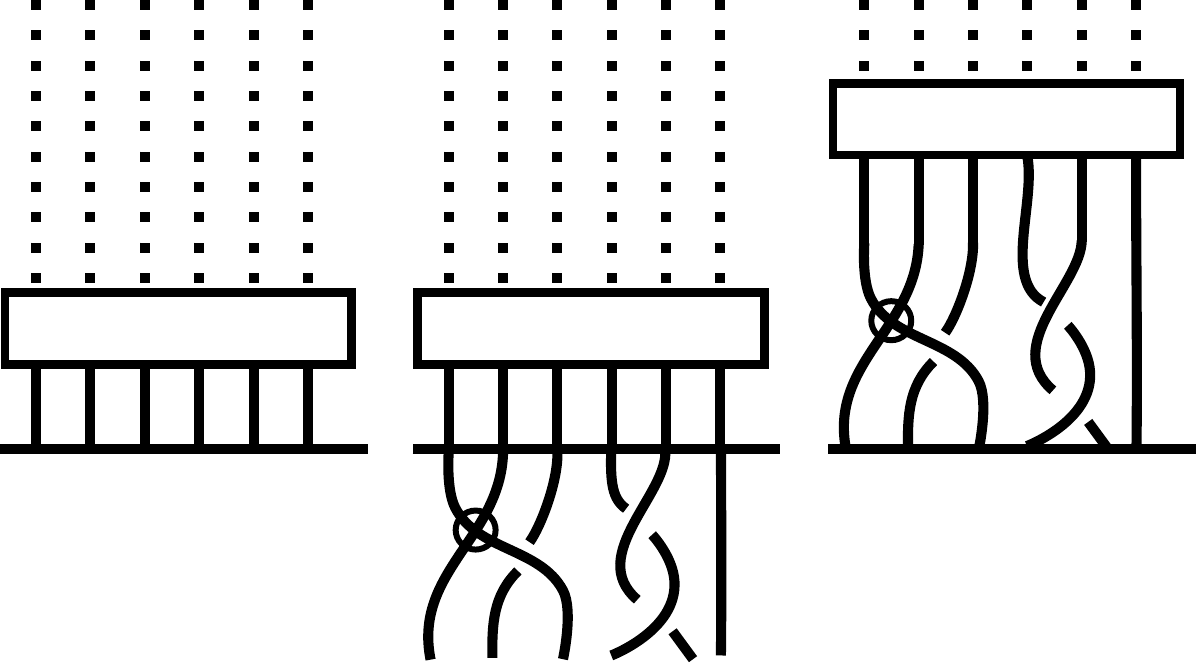}
  \caption{Stacking a virtual braid on a graph tangle.}
  \label{fig:stack}
\end{figure}

\begin{figure}[htbp]
  \centering
  \includegraphics[width=.16\linewidth]{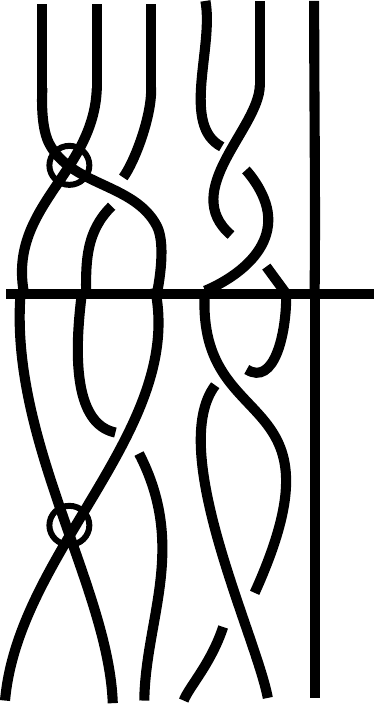}
  \caption{A braid followed by its reverse.}
  \label{fig:reverse-braid}
\end{figure}

\begin{defn}\label{def:trivial-tangle}
A virtual graph tangle $T$ is \emph{trivial} if every component of
$T$ is a tree or a proper arc and there is a virtual braid $\beta$
such that stacking $T$ with $\beta$ produces a crossingless tangle.
\end{defn}

The requirement that the components be trees or proper arcs is
essential.  In particular, a closed cycle contained in one half-plane
is not counted as a component of a trivial tangle.

\begin{defn}
Let $D$ be a virtual spatial graph diagram.  A horizontal line $\ell$
disjoint from the crossings and vertices is a \emph{bridge line} if
it cuts $D$ into two nonempty trivial tangles $T_+$ and $T_-$.  We
write $D=T_+\cup_\ell T_-$ and call this a \emph{bridge splitting}.
If $D$ is disconnected, every component of $D$ is required to meet
$\ell$.
\end{defn}

\begin{prop}\label{prop:existence}
Every virtual spatial graph has a bridge splitting.
\end{prop}

\begin{proof}
Choose a generic height function on a diagram $D$.  By the usual
pulling-up and pulling-down procedure, move every local maximum above
all crossings and every local minimum below all crossings.  Move each
vertex either above or below the crossing region.  Passing a critical
point or vertex through another strand is accomplished by a
classical or virtual Reidemeister move; the local steps are indicated
in Figure~\ref{fig:pulling}.  Consequently, all crossings lie in one
horizontal strip, while the portions above and below that strip are
crossingless graphs followed by virtual braids.

Choose a horizontal line through the crossing strip.  If one side
contains a cycle, select an open subarc of that cycle and move it
across the line.  In the virtual category the new route is made with
detour moves; in the classical category the same operation is an
ambient isotopy in three-space.  On the original side this cuts the
cycle, and on the other side it adds a proper arc.  Each such operation
reduces the first Betti number of the graph on the original side.
After finitely many operations, both sides are forests together with
proper arcs.  The same operation ensures that every component of $D$
meets the line.

Each side is now a crossingless forest followed by a virtual braid.
Stacking the reverse braid cancels the braid crossings, so each side
is a trivial tangle in the sense of
Definition~\ref{def:trivial-tangle}.
\end{proof}

\begin{figure}[htbp]
  \centering
  \includegraphics[width=.5\linewidth]{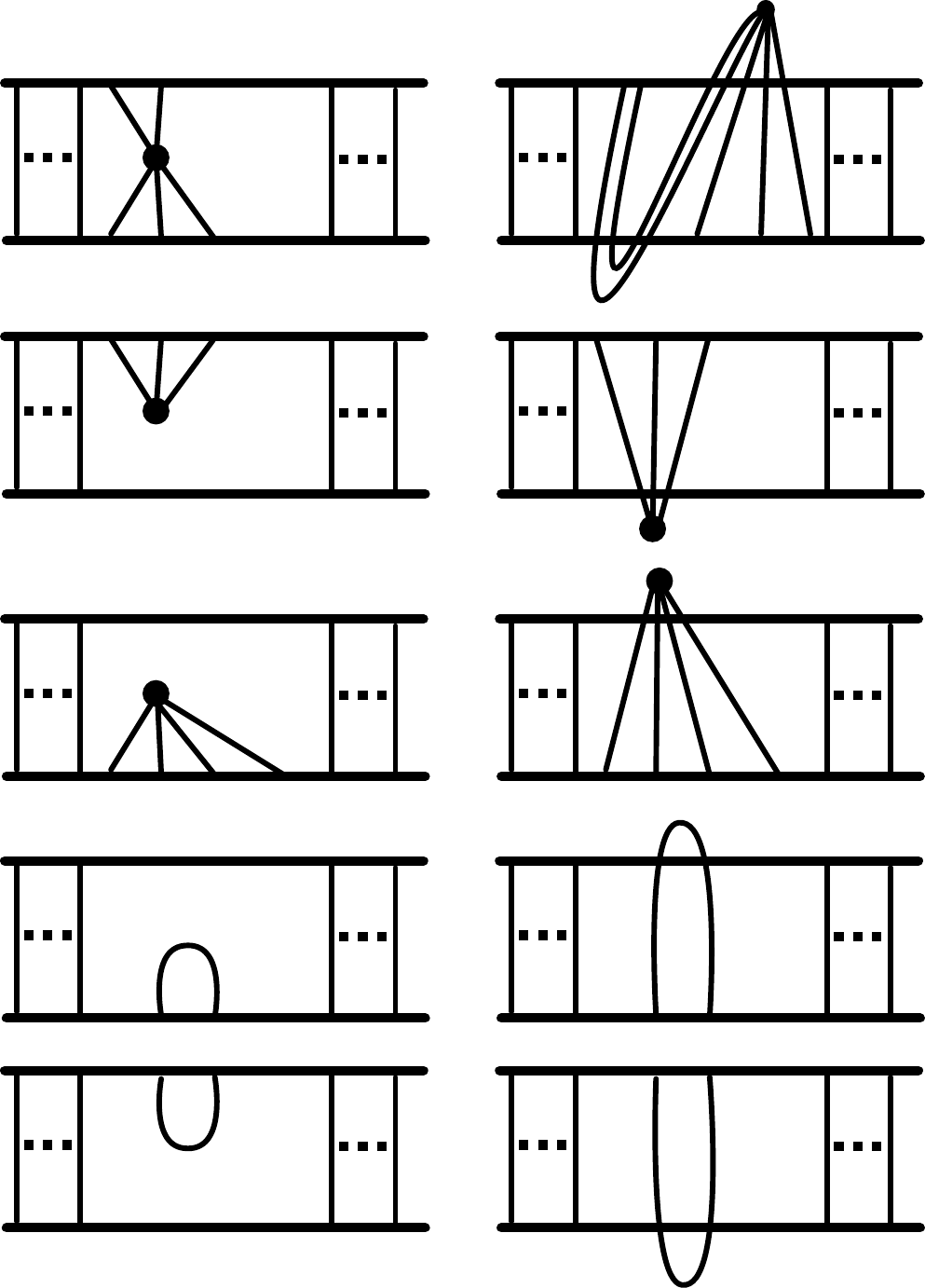}
  \caption{Moving extrema and vertices out of the crossing region.}
  \label{fig:pulling}
\end{figure}

For a component $P$ of a trivial tangle, let
$\omega(P)=|P\cap\ell|$ be its \emph{weight}.  The next elementary
formula is the bookkeeping used throughout the paper.

\begin{lem}[Component-weight formula]\label{lem:weight}
Let $P$ be a connected tree component of a trivial tangle, and let
$V(P)$ be the set of graph vertices in its interior.  Then
\begin{equation}\label{eq:component-weight}
       \omega(P)=2+\sum_{v\in V(P)}\bigl(\deg(v)-2\bigr).
\end{equation}
The same formula applies to a proper arc, with the empty sum.
Consequently, if $P_1,\ldots,P_s$ are the components of a trivial
tangle and together contain a set $W$ of graph vertices, then
\begin{equation}\label{eq:forest-weight}
 \sum_{i=1}^s\omega(P_i)
   =2s+\sum_{v\in W}\bigl(\deg(v)-2\bigr).
\end{equation}
\end{lem}

\begin{proof}
Regard the $\omega(P)$ boundary points as degree-one vertices of the
finite tree $P$.  If $I=|V(P)|$, then $P$ has
$I+\omega(P)-1$ edges.  The handshaking lemma gives
\[
 \sum_{v\in V(P)}\deg(v)+\omega(P)
       =2\bigl(I+\omega(P)-1\bigr).
\]
Rearranging proves the first formula.  Summing it over the components
proves the second.  A proper arc has two boundary points, so it is the
case in which $V(P)$ is empty.
\end{proof}

Since all graph vertices have valence at least two, every trivial
tangle component has weight at least two.

\begin{defn}
The \emph{bridge index} (or \emph{bridge number}) of a virtual spatial graph $G$
is
\[
 b(G)=\min_{D=T_+\cup_\ell T_-}|D\cap\ell|,
\]
where the minimum is over all bridge splittings of all diagrams of
$G$.  The \emph{constrained bridge index} $b_c(G)$ is the same
minimum with the additional requirement that all graph vertices lie
in the same one of the two tangles.
\end{defn}
\begin{rem}
We count all bridge-surface intersections.  Authors who divide this
number by two recover the usual classical bridge number on knots and
links.  In particular, our $b$ is twice the bridge index denoted by
$\beta$ in \cite{blackwell2026}.
\end{rem}
Plainly $b(G)\leq b_c(G)$.  For a virtual knot, $b(G)$ is twice the second
bridge index of Nakanishi and Satoh \cite{nakanishiSatoh}.  The
classical spatial-graph definition is closely related to bridge
positions studied by Ozawa and by Taylor--Tomova
\cite{ozawa2012,taylorTomova}.

\begin{defn}\label{def:virtual-handlebody}
Two virtual spatial graph diagrams are \emph{virtual-handlebody
equivalent} if they are related by virtual spatial-graph Reidemeister
moves, IH moves, and the following vertex moves and their inverses:
an $n$-valent vertex may be expanded into a planar trivalent tree with
the same cyclically ordered $n$ leaves, or such a tree may be
contracted to an $n$-valent vertex.  A \emph{virtual handlebody-link}
is an equivalence class for these moves; write $[G]_{\mathrm{hb}}$ for
the class represented by $G$.  For such a class $H$, define
\[
 b_{\mathrm{hb}}(H)
   =\min\{\,b(G') : G'\text{ is a trivalent representative of }H\,\}.
\]
In the classical category, we use the same notation for the usual
handlebody-link represented by the regular neighborhood $N(G)$.
\end{defn}

\begin{lem}[Vertex expansions in bridge splittings]
\label{lem:handlebody-tangles}
Every bridge splitting admits a trivalent expansion of its vertices
with the same intersection count and the same component counts.
More generally, a planar tree replacement made in either
crossing-free forest of its braid presentation preserves these
counts and the triviality of the tangles.
\end{lem}

\begin{proof}
Use the trivializing braids to express the splitting as two
crossing-free forests separated by a braid.  At each vertex choose
a small disk meeting its forest in a pod, and replace that pod by a
planar trivalent tree with the same boundary leaves.  This changes
neither the connections to the rest of the component nor its first
Betti number: a replacement with $t$ vertices adds $t-1$ vertices
and $t-1$ edges.  Each component therefore remains a tree, and the
forest remains crossing-free.  No boundary endpoint is changed.

Put back the same middle braid.  The resulting tangles are trivial
by definition and have the original component and intersection
counts.  The same argument applies to a planar IH replacement or
contraction in a small disk of either crossing-free forest.
\end{proof}

\begin{cor}\label{cor:handlebody-comparison}
For every virtual spatial graph $G$,
\[
                         b_{\mathrm{hb}}([G]_{\mathrm{hb}})\leq b(G).
\]
In the classical category this reads
\[
                         b_{\mathrm{hb}}(N(G))\leq b(G).
\]
\end{cor}

\begin{proof}
Take a bridge splitting of $G$ realizing $b(G)$ and expand every
higher-valence vertex inside the tangle that contains it.  By
Lemma~\ref{lem:handlebody-tangles}, the resulting trivalent
representative $G'$ has a bridge splitting with no more intersection
points than the original one.  Since $G'$ represents
$[G]_{\mathrm{hb}}$, minimizing over all representatives gives the
result.  The expansions are made in the crossing-free forests
before restoring their trivializing braids.
\end{proof}

\begin{defn}[Unweighted bridge number]\label{def:unweighted-bridge}
For a bridge splitting $D=T_+\cup_\ell T_-$ of a virtual spatial
graph $\Gamma$, set
\[
 \widehat b(D,\ell)
   =\min\{\,|\pi_0(T_+)|,|\pi_0(T_-)|\,\}.
\]
Define $\widehat b(\Gamma)$ by minimizing over bridge splittings of
diagrams of $\Gamma$, and define
\[
 \widehat b(H)
   =\min\{\,\widehat b(\Gamma):
                \Gamma\text{ is a trivalent representative of }H\,\}
\]
for a virtual handlebody-link $H$.  Circle components with no vertices
are allowed in a trivalent representative.
\end{defn}

Write $\chi(H)=\chi(\Gamma)$ for any spine
representing $H$; the allowed spine moves preserve this number.
If $r=|\pi_0(T_+)|$, $s=|\pi_0(T_-)|$, and $m=|D\cap\ell|$, then
Euler characteristic gives
\[
 \chi(H)=r+s-m.
\]
Consequently,
\[
 2\widehat b(D,\ell)\leq r+s
                    =m+\chi(H).
\]
Applying this to a splitting that realizes $b_{\mathrm{hb}}(H)$ gives
\begin{equation}\label{eq:unweighted-half-intersection}
 b_{\mathrm{hb}}(H)\geq2\widehat b(H)-\chi(H).
\end{equation}

\section{Flowed families of virtual biquandle colorings}
\label{sec:coloring}

We next formulate a coloring invariant that retains both a flow label
and a nontrivial rule at a virtual crossing.  The classical part is
the $G$-family formalism of Ishii--Nelson and Murao
\cite{ishiiNelson,murao2018}.  We require the virtual crossing rule
to be compatible with the vertex rule as well as the virtual-link
moves \cite{fenn2004,kauffmanManturov}.  We
write the flow group as $A$ to avoid confusing it with a spatial
graph.

\begin{defn}
Let $A$ be a finite abelian group, written additively.  An
\emph{$A$-family of biquandles} is a nonempty set $X$ with operations
\[
 \ustar^{\,a},\ostar^{\,a}:X\times X\longrightarrow X
 \qquad(a\in A)
\]
such that every sideways switch
\[
 S_{a,b}(x,y)=\bigl(y\ostar^{\,a}x,x\ustar^{\,b}y\bigr)
\]
is a bijection, each right translation
$x\mapsto x\ustar^{\,a}y$ and $x\mapsto x\ostar^{\,a}y$ is a
bijection, and the following identities hold for all $a,b\in A$
and $x,y,z\in X$:
\begin{align}
 x\ustar^{\,a}x&=x\ostar^{\,a}x,                                      \tag{F1}\label{eq:F1}\\
 x\ustar^{\,a+b}y
  &=(x\ustar^{\,a}y)\ustar^{\,b}(y\ustar^{\,a}y),                    \tag{F2u}\label{eq:F2u}\\
 x\ostar^{\,a+b}y
  &=(x\ostar^{\,a}y)\ostar^{\,b}(y\ostar^{\,a}y),                  \tag{F2o}\label{eq:F2o}\\
 x\ustar^{\,0}y&=x,\qquad x\ostar^{\,0}y=x,                         \tag{F0}\label{eq:F0}
\end{align}
and the three exchange laws
\begin{align}
 (x\ustar^{\,a}y)\ustar^{\,b}(z\ostar^{\,a}y)
   &=(x\ustar^{\,b}z)\ustar^{\,a}(y\ustar^{\,b}z),                 \tag{E1}\label{eq:E1}\\
 (x\ostar^{\,a}y)\ustar^{\,b}(z\ostar^{\,a}y)
   &=(x\ustar^{\,b}z)\ostar^{\,a}(y\ustar^{\,b}z),                 \tag{E2}\label{eq:E2}\\
 (x\ostar^{\,a}y)\ostar^{\,b}(z\ostar^{\,a}y)
   &=(x\ostar^{\,b}z)\ostar^{\,a}(y\ustar^{\,b}z).                 \tag{E3}\label{eq:E3}
\end{align}
\end{defn}

\begin{lem}[Diagonal maps]\label{lem:diagonal-maps}
For $a\in A$, put
\[
 \kappa_a(x)=x\ustar^{\,a}x=x\ostar^{\,a}x.
\]
Then $\kappa_{a+b}=\kappa_b\circ\kappa_a$,
$\kappa_0=\operatorname{id}_X$, and
$\kappa_a^{-1}=\kappa_{-a}$.
\end{lem}

\begin{proof}
Set $y=x$ in \eqref{eq:F2u}, and use \eqref{eq:F1} and
\eqref{eq:F0}.  Taking $b=-a$ and then interchanging $a$ and $b$
gives both inverse identities.
\end{proof}

\begin{defn}[Virtual crossing compatibility]\label{def:virtual-family}
Set $Q=X\times A$.  To distinguish the modern sideways convention
from a braid convention, write the incoming-pair crossing map as
\[
 \mathcal R((x,a),(z,b))=((y,b),(x\ustar^{\,b}y,a)),
 \qquad z=y\ostar^{\,a}x.
\]
The right-translation axiom determines $y$ uniquely from $x,z$.
Thus $x$ is the incoming under-color, $z$ is the incoming over-color,
and $y$ is the outgoing over-color.  The partially defined vertex
multiplication is
\[
 \mu((x,a),(\kappa_a(x),b))=(x,a+b).
\]
An \emph{$A$-family of virtual biquandles}, for use with the
handlebody moves in this paper, is an $A$-family of biquandles
equipped with a virtual switch $\mathcal V:Q^2\to Q^2$ that swaps
the two flow labels.  Write $\mathcal V=(\mathcal V_1,\mathcal V_2)$.
Nondegeneracy means that, for each fixed $p\in Q$, the maps
$q\mapsto\mathcal V_1(p,q)$ and
$q\mapsto\mathcal V_2(q,p)$ are bijections.  Require
$\mathcal V^2=\operatorname{id}_{Q^2}$ and require both coordinate
projections $\operatorname{Fix}(\mathcal V)\to Q$ to be bijective;
the latter is the unique-extension relation for a virtual curl.
For a map $F:Q^2\to Q^2$, set
$F_{12}=F\times\operatorname{id}_Q$ and
$F_{23}=\operatorname{id}_Q\times F$.  The other virtual-link
relations are
\[
 \mathcal V_{12}\mathcal V_{23}\mathcal V_{12}
   =\mathcal V_{23}\mathcal V_{12}\mathcal V_{23},
 \qquad
 \mathcal V_{12}\mathcal V_{23}\mathcal R_{12}
   =\mathcal R_{23}\mathcal V_{12}\mathcal V_{23},
\]
with maps composed from right to left.  Reversed strand orientations
use the equations solved for the corresponding pair of semiarcs.

In addition, the virtual switch must satisfy the vertex identities
\begin{align}
 \mathcal V\circ(\mu\times\operatorname{id})
   &=(\operatorname{id}\times\mu)
      \circ(\mathcal V\times\operatorname{id})
      \circ(\operatorname{id}\times\mathcal V),
       \label{eq:virtual-vertex-left}\\
 \mathcal V\circ(\operatorname{id}\times\mu)
   &=(\mu\times\operatorname{id})
      \circ(\operatorname{id}\times\mathcal V)
      \circ(\mathcal V\times\operatorname{id}),
       \label{eq:virtual-vertex-right}
\end{align}
as equalities of partially defined maps, including equality of their
domains, and with the corresponding identities for the reversed
Y-orientations.  Thus passing a strand virtually across a vertex
preserves colorings, with unique extension and unchanged exterior
colors.  Taking $\mathcal V(p,q)=(q,p)$ always gives such a virtual
extension.
\end{defn}

The usual virtual-link switch
$V(x,y)=(v^{-1}(y),v(x))$ associated with an automorphism $v$ does
not automatically satisfy the vertex identities.  In particular,
$v^3=\operatorname{id}_X$ is not a substitute for
\eqref{eq:virtual-vertex-left}--\eqref{eq:virtual-vertex-right}.
The identity choice $v=\operatorname{id}_X$, used in the examples
below, gives the flip switch and satisfies these identities.
Indeed, for a vertex with zero flow on all three incident edges,
$\mu((x,0),(x,0))=(x,0)$.  Using the same automorphism switch at
every virtual crossing, the two sides of
\eqref{eq:virtual-vertex-left} give $v^{-1}(z)$ and $v^{-2}(z)$
on the passing strand.  Equality for every $z$ forces
$v=\operatorname{id}_X$.  A nontrivial virtual extension therefore
requires a different compatible switch.

A \emph{Y-orientation} directs every edge so that no trivalent
vertex is a source or a sink.  An $S^1$-orientation of a classical
handlebody-link specifies orientations only on the cores of its
genus-one components; a component of higher genus carries no
additional orientation data \cite{murao2018}.  A Y-oriented spine
represents this data through its oriented circle components.
We use the corresponding convention for virtual handlebody classes.
Thus the connected examples of genus at least two below carry no
additional orientation restriction in the minimum defining their
handlebody bridge indices.

\begin{defn}
Let $D$ be a Y-oriented virtual trivalent graph diagram.  An
\emph{$A$-flow} is an assignment $\phi$ of an element of $A$ to every
oriented edge such that at each vertex $w$,
\[
 \sum_{e\in E_{\rm in}(w)}\phi(e)
       =\sum_{e\in E_{\rm out}(w)}\phi(e).
\]
The label follows its edge through classical and virtual crossings.
We denote the set of flows by
\[
                              \Flow(D;A).
\]
\end{defn}

Cut the diagram at every crossing and vertex; the resulting pieces
are its semi-arcs.  Fix an $A$-flow $\phi$.  At the reference positive
classical crossing, let the incoming under-semi-arc have color $x$
and flow $a$, and let the \emph{outgoing} over-semi-arc have color
$y$ and flow $b$.  The remaining colors are
\[
 \text{outgoing under: }x\ustar^{\,b}y,
 \qquad
 \text{incoming over: }y\ostar^{\,a}x.
\]
Equivalently, an incoming pair is propagated by $\mathcal R$ above.
This is the modern sideways convention of
\cite{ishiiNelson}; both over-semi-arcs must be retained.
The opposite crossing uses the inverse incoming-pair map, and other
orientations are obtained by solving the same local equations.
At a Y-oriented
vertex with incoming flows $a,b$ and outgoing flow $a+b$, the three
colors, in the order left incoming, right incoming, outgoing, are
\[
                  x,\qquad x\ostar^{\,a}x,\qquad x.
\]
The reverse Y-orientation uses the same rule backward.  Finally, at a
virtual crossing the pair of color-and-flow labels changes by
$\mathcal V$.
An assignment satisfying these rules is an \emph{$X$-coloring} of
$(D,\phi)$; its set is denoted
$\operatorname{Col}_{X}(D,\phi)$, with $\mathcal V$ understood.

\begin{prop}\label{prop:virtual-invariance}
Let $X$ be a finite $A$-family of virtual biquandles in the sense of
Definition~\ref{def:virtual-family}.  Every sequence of oriented
classical and virtual graph Reidemeister moves and IH moves induces
a bijection of flow sets and, for each pair of corresponding flows,
a bijection of fixed-flow coloring sets.
\end{prop}

\begin{proof}
For the classical crossing and vertex moves, the required bijections
are those obtained from the partially multiplicative biquandle in
\cite[Propositions~5.2 and~5.4]{ishiiNelson}: the Reidemeister I
and II moves use
\eqref{eq:F1}--\eqref{eq:F0} and invertibility of the switches, and
the Reidemeister III move uses the three exchange laws
\eqref{eq:E1}--\eqref{eq:E3}.  The same cited propositions verify
both sliding a strand past a classical vertex and twisting two
adjacent vertex branches.  These checks use the partial
multiplication $\mu$, including its domain.

The virtual and mixed relations in
Definition~\ref{def:virtual-family} give the same unique-extension
property for virtual and mixed crossing moves.  The two vertex
identities give it for the virtual vertex moves.

For an IH move, the two trivalent trees encode the two
parenthesizations of the same partially defined product $\mu$.
Their admissible boundary labels agree because
$\kappa_{a+b}=\kappa_b\circ\kappa_a$.  On each side, the internal
flow label is prescribed by the boundary flow labels, and the
internal $X$-label is prescribed by the vertex relation.  This gives
the required bijections for IH moves.  Expansions and contractions
of a planar vertex tree are treated by iterating these products:
the coloring of a higher-valence vertex is interpreted through such
a tree.

At ordinary Reidemeister moves, the flow labels extend uniquely.
At virtual crossings they simply follow the strands.  Hence all the
coloring bijections just described lie over the corresponding
bijections of flow sets.  The bijection associated with a sequence
need not be independent of the chosen sequence; its existence is
sufficient for fixed-flow cardinalities and their maximum to be
invariant.
\end{proof}

For a finite flow group, forgetting the chosen flow gives the multiset
invariant
\[
 \Phi_X(D)=
 \left\{\!\left\{
   \#\operatorname{Col}_{X}(D,\phi)
   \ \middle|\ \phi\in\Flow(D;A)
 \right\}\!\right\}.
\]

\begin{rem}[Higher-valence vertices]\label{rem:higher-valence}
A higher-valence vertex can be replaced by a Y-oriented trivalent
tree, after which the preceding construction applies.  Different
trivalent trees are related by IH moves, and the ordinary $A$-family
theory gives equal coloring counts under those moves.  This produces
an invariant of the corresponding virtual \emph{handlebody-link},
whose spine is defined only up to IH moves.  It should not be confused
with an invariant of the original spatial graph: a trivalent
expansion preserves the regular neighborhood, but need not preserve
the spatial-graph isotopy type.  The bridge comparison in
Corollary~\ref{cor:handlebody-comparison} is nevertheless available
for this handlebody equivalence.
\end{rem}

\begin{rem}[Parallel operations]
Let $(X,\ustar,\ostar)$ be a finite biquandle.  Define
parallel operations recursively by
\begin{align*}
 x\ustar^{[0]}y&=x,&x\ustar^{[1]}y&=x\ustar y,\\
 x\ostar^{[0]}y&=x,&x\ostar^{[1]}y&=x\ostar y,
\end{align*}
and
\begin{align*}
 x\ustar^{[i+j]}y
   &=(x\ustar^{[i]}y)\ustar^{[j]}(y\ustar^{[i]}y),\\
 x\ostar^{[i+j]}y
   &=(x\ostar^{[i]}y)\ostar^{[j]}(y\ostar^{[i]}y).
\end{align*}
The \emph{parallel type} $\type(X)$ is the least positive integer
$N$ for which $x\ustar^{[N]}y=x\ostar^{[N]}y=x$ for all $x,y$.
It is a period of these operations, rather than the cardinality
of $X$.  If $N=\type(X)$, these operations are indexed by $\Z_N$ and form a
$\Z_N$-family.  Equipping it with the flip virtual switch produces a
$\Z_N$-family of virtual biquandles.  Any other virtual switch must
also satisfy Definition~\ref{def:virtual-family}.
\end{rem}

\subsection{Bridge bounds from fixed-flow colorings}
\label{subsec:flowed-bridge-bound}

The counting argument below is the bridge analogue of the classical
coloring estimates for handlebody-knots; compare
\cite{ishiiNelson,murao2018}.  Its essential feature is that one
$X$-color determines a colored trivial tree once its flow is fixed.
This does not require the three vertex colors to be equal.

We allow a finite group $G$ in this subsection, using the ordinary
$G$-family and $G$-flow conventions of
\cite{ishiiNelson,murao2018}.  For a nonabelian group, the flow on an
underpassing arc changes from $g$ to $h^{-1}gh$ when the overpassing
arc has flow $h$, and the flow relation at a vertex is an ordered
product.  The associated multiplication is
\[
 \mu((x,g),(\kappa_g(x),h))=(x,gh),
 \qquad
 \kappa_g(x)=x\ustar^{\,g}x=x\ostar^{\,g}x.
\]
The group-family identities give
\[
 \kappa_{gh}=\kappa_h\circ\kappa_g,\qquad
 \kappa_{1_G}=\operatorname{id}_X,\qquad
 \kappa_g^{-1}=\kappa_{g^{-1}}.
\]
The virtual switch is required to swap the flow labels, satisfy the
virtual and mixed relations with the associated classical switch,
and satisfy the vertex compatibility
\eqref{eq:virtual-vertex-left}--\eqref{eq:virtual-vertex-right}
and its oriented forms.  For the abelian group $A$ used elsewhere
in this paper, these are precisely the preceding definitions.
Write $\Flow_G(D)$ for the set of flows; thus
$\Flow_A(D)=\Flow(D;A)$.

\begin{lem}[One color determines a tree]\label{lem:flowed-tree}
Let $T$ be a crossing-free connected tree tangle or a proper arc,
with a fixed $G$-flow.  Fix a semiarc $e$ of $T$.  An $X$-coloring
of $T$ is uniquely determined by its color on $e$.
\end{lem}

\begin{proof}
At a standard Y-oriented vertex, the three colors have the form
\[
                         x,\quad\kappa_g(x),\quad x,
\]
where $g$ is prescribed by the flow.  Since $\kappa_g$ is a
bijection, knowing any one of these three colors determines the
other two.  The same holds for the reversed vertex rule.  Root the
tree at a point of $e$ and propagate along its edges.  The unique
paths in the tree determine every remaining color.  A proper arc
has a single color and needs no vertex argument.
\end{proof}
The following bound is in terms of unweighted bridge index, but the actual weights can be added later depending on the spatial graph of interest.
\begin{thm}[Fixed-flow bridge bound]\label{thm:g-family-bound}
Let $H$ be a virtual handlebody-link, and let $D$ be a Y-oriented
trivalent diagram representing $H$.  Fix orientations on any circle
components, and use the corresponding oriented handlebody moves.
Let $G$ be a finite group, and let $X$ be a finite $G$-family of
virtual biquandles with the compatibility specified above and
$|X|>1$.  Then, for every $\rho\in\Flow_G(D)$,
\begin{equation}\label{eq:fixed-flow-bridge-count}
       |\operatorname{Col}_X(D,\rho)|\leq |X|^{\widehat b(H)}.
\end{equation}
With the convention $\log_{|X|}0=-\infty$,
\begin{equation}\label{eq:g-family-unweighted-bound}
 \widehat b(H)\geq
 \max_{\rho\in\Flow_G(D)}
       \log_{|X|}|\operatorname{Col}_X(D,\rho)|.
\end{equation}

\end{thm}

\begin{proof}
Fix $\rho\in\Flow_G(D)$.  The flowed handlebody Reidemeister and
IH moves induce bijections of flow sets and bijections of the
coloring sets over corresponding flows, as in
Proposition~\ref{prop:virtual-invariance}.  For a nonabelian flow
group this is the usual $G$-family invariance
\cite{ishiiNelson,murao2018}, together with the assumed virtual
compatibility.  Vertex expansions and contractions are interpreted
through their trivalent trees.  Thus we may transport $\rho$ and
its coloring set to any trivalent representative of $H$ along a
chosen sequence of these moves.

Choose a trivalent representative $\Gamma$ and a bridge splitting
that realizes $\widehat b(H)$.  Since both tangles are trivial, their
trivializing virtual braids put the diagram into the form of a
crossing-free upper forest, a virtual braid in a middle strip, and
a crossing-free lower forest.  This changes neither the number of
components on either side nor the fixed-flow coloring cardinality.
Denote the resulting diagram and transported flow by $D_S$ and
$\rho_S$.  In the classical setting this is the diagram adapted to
a bridge sphere in $S^3$; in the virtual setting it is a bridge-line
presentation.

Interchange the upper and lower sides if necessary, so that the
upper forest has the smaller number of components.  Write these
components as $T_1,\ldots,T_r$.  By the choice of the splitting,
\[
                         r=\widehat b(H).
\]
Choose a semiarc $e_i$ in each $T_i$.  By
Lemma~\ref{lem:flowed-tree}, the colors on $e_1,\ldots,e_r$
determine every color in the upper forest, and hence all colors
entering the middle braid.

The flow on every semiarc of the braid is already prescribed by
$\rho_S$.  At a classical crossing use the incoming-pair map
$\mathcal R$, its inverse, or the corresponding oriented form,
according to the crossing.  The sideways switch $S_{g,h}$ is used
to solve the local equations when the chosen inputs are a sideways
pair.  These maps are bijective by the biquandle axioms.
The virtual switch and its oriented forms are likewise
bijective.  Reading the braid from top to bottom therefore
determines every color in it and all boundary colors of the lower
forest.

Each lower component is a tree or a proper arc and meets the
boundary.  By Lemma~\ref{lem:flowed-tree}, its prescribed boundary
colors extend to at most one coloring of the component.  They may
be incompatible with its vertex relations, but they introduce no
further choices.  Therefore the evaluation map
\[
 \operatorname{Col}_X(D_S,\rho_S)\longrightarrow X^r,\qquad
 C\longmapsto(C(e_1),\ldots,C(e_r))
\]
is injective.  It follows that
\[
 |\operatorname{Col}_X(D,\rho)|
   =|\operatorname{Col}_X(D_S,\rho_S)|
   \leq |X|^r
   =|X|^{\widehat b(H)}.
\]
This proves \eqref{eq:fixed-flow-bridge-count} for every fixed
flow.  Taking logarithms for nonempty coloring sets and then the
maximum over flows proves \eqref{eq:g-family-unweighted-bound}.
\end{proof}

\begin{cor}[Alexander bridge bound]
\label{cor:alexander-unweighted-bound}
Under the hypotheses of Theorem~\ref{thm:g-family-bound}, suppose
$X$ is a $d$-dimensional vector space over a finite field $\mathbb F$,
where $d\geq1$.  Suppose each operation
$\ustar^{\,g},\ostar^{\,g}:X\oplus X\to X$ is $\mathbb F$-linear,
and, for every prescribed pair of flow labels, the two-color
virtual switch $X\oplus X\to X\oplus X$ is also $\mathbb F$-linear.
Then
\begin{equation}\label{eq:alexander-unweighted-bound}
 \widehat b(H)\geq
 \frac1d\max_{\rho\in\Flow_G(D)}
            \dim_{\mathbb F}\operatorname{Col}_X(D,\rho).
\end{equation}
\end{cor}

\begin{proof}
Fix $\rho$.  With the flow labels prescribed, all classical,
virtual, and vertex equations are homogeneous linear equations
over $\mathbb F$.  In particular,
$\kappa_g(x)=x\ostar^{\,g}x$ is linear.  Thus
$\operatorname{Col}_X(D,\rho)$ is an $\mathbb F$-vector subspace of
the space of semiarc assignments.

Write $q=|\mathbb F|$ and
$k_\rho=\dim_{\mathbb F}\operatorname{Col}_X(D,\rho)$.
Then $|X|=q^d$ and
$|\operatorname{Col}_X(D,\rho)|=q^{k_\rho}$.
Theorem~\ref{thm:g-family-bound} gives
\[
 q^{k_\rho}\leq(q^d)^{\widehat b(H)},
 \qquad\text{hence}\qquad
 k_\rho\leq d\,\widehat b(H).
\]
Divide by $d$ and maximize over $\rho$ to obtain
\eqref{eq:alexander-unweighted-bound}.  
\end{proof}

Every individual fixed-flow count gives a bridge lower bound;
enumerating all flows is only needed to find the strongest bound
from the chosen family.  The sum of the counts over all flows
cannot replace the fixed-flow count in
\eqref{eq:fixed-flow-bridge-count}: allowing the flow itself to
vary introduces additional choices.

\subsection{Quandle specializations}

A \emph{quandle} is a set $Q$ with an operation $\triangleright$ such
that (i) $x\triangleright x=x$; (ii) each right translation
$S_y(x)=x\triangleright y$ is a bijection; and (iii)
\[
 (x\triangleright y)\triangleright z
   =(x\triangleright z)\triangleright(y\triangleright z).
\]
It becomes a one-sided biquandle by setting
$x\ustar y=x\triangleright y$ and $x\ostar y=x$.  If $v$ is a quandle
automorphism, $V(x,y)=(v^{-1}(y),v(x))$ is the usual virtual-link
rule.  In the flowed handlebody setting one must additionally check
the vertex compatibility in Definition~\ref{def:virtual-family};
the flip switch always satisfies it.

The \emph{type} of a finite quandle is the least positive integer
$N$ such that $S_y^N=\operatorname{id}$ for every $y$.  On an
oriented graph, define the signed valence
\[
 \epsilon(v)=\#E_{\rm in}(v)-\#E_{\rm out}(v).
\]
The \emph{equal-vertex rule} requires all incident semi-arcs to
have the same color.  Its compatibility condition concerns signed
valence, rather than unsigned valence alone.

\begin{prop}\label{prop:equal-vertex}
Let $Q$ be a finite quandle of type $N$, and give $G$ an orientation
such that $N$ divides $\epsilon(v)$ at every vertex.  The number of
quandle colorings with the equal-vertex rule and unchanged colors
at virtual crossings is an invariant of the oriented virtual
spatial graph.
\end{prop}

\begin{proof}
The three quandle axioms give the usual unique-extension bijections
for the classical crossing moves.  At a vertex twist all involved
colors are equal, so idempotency applies.  When a strand passes
across a vertex colored $y$, its successive changes are right
translations $S_y$ or $S_y^{-1}$ according to the orientations.
The total change is $S_y^{\epsilon(v)}$ or its inverse, hence is
the identity.  The other colors are preserved by the same quandle
relations.  Virtual crossing and virtual vertex moves simply carry
colors along their strands.  All these correspondences are
reversible and fix the exterior colors.
\end{proof}

We use two examples.  The dihedral quandle
\[
 R_p=\Z_p,\qquad x\triangleright y=2y-x\pmod p,
\]
has type two when $p\geq3$ is odd.  In particular, $R_3$ can be used with
the equal-vertex rule on graphs of even valence.  The tetrahedral
quandle $T$ has operation table
\[
\begin{array}{c|cccc}
\triangleright&0&1&2&3\\ \hline
0&0&2&3&1\\
1&3&1&0&2\\
2&1&3&2&0\\
3&2&0&1&3
\end{array}
\]
and type three.  It can therefore be used with the equal-vertex
rule on trivalent graphs oriented so that every vertex is a source
or a sink.  An arbitrary Y-orientation with all labels equal to $1$
would not meet the signed-valence condition.  Equivalently,
\[
 ((x\triangleright y)\triangleright y)\triangleright y=x
 \qquad(x,y\in T).
\]

For an Eulerian spatial graph, Ishii and Yasuhara use a different
vertex relation \cite{ishii1997}.  With $R_p$, list the incident
semi-arcs in cyclic order and require their alternating sum to vanish.  We
use that rule only in the calculation of the next section; the bridge
bound in Section~\ref{sec:gap} uses the equal-vertex rule.

\subsection{A biquandle of interest}
\label{subsec:kishino-graph}

Write $T_4$ for the four-element biquandle with both $\overline{\triangleright}$ and $\underline{\triangleright}$ operations equal to the
following operation $\circ$:
\[
\begin{array}{c|cccc}
\circ&1&2&3&4\\ \hline
1&1&3&2&4\\
2&2&4&1&3\\
3&4&2&3&1\\
4&3&1&4&2
\end{array}
\]
The associated diagonal map $w(x)=x\circ x$ is
\[
 w(1)=1,\qquad w(2)=4,\qquad w(3)=3,\qquad w(4)=2.
\]
In particular, $w^2=\operatorname{id}$.  These are modern sideways
operation tables; a row is the first argument and a column is the
second.  They are not the incoming-pair operation tables in
Nelson--Vo's original convention \cite{nelsonVo2006}.

\begin{lem}
Let \(X=\{1,2,3,4\}\) and define
\[
x\mathbin{\underline{\triangleright}}y
=
x\mathbin{\overline{\triangleright}}y
=
x\circ y
\]
by the table above. Then
\((X,\mathbin{\underline{\triangleright}},
\mathbin{\overline{\triangleright}})\) is a biquandle.
\end{lem}

\begin{proof}
Identify \(X\) with \(\mathbb{F}_2^2\) by
\[
1=\binom{0}{0},\qquad
2=\binom{0}{1},\qquad
3=\binom{1}{0},\qquad
4=\binom{1}{1},
\]
where elements of \(\mathbb{F}_2^2\) are regarded as column vectors.
A direct comparison with the operation table shows that
\[
x\circ y=Ax+By,
\qquad
A=
\begin{pmatrix}
1&0\\
1&1
\end{pmatrix},
\qquad
B=
\begin{pmatrix}
0&1\\
1&0
\end{pmatrix}.
\]
Equivalently,
\[
\binom{x_1}{x_2}\circ\binom{y_1}{y_2}
=
\binom{x_1+y_2}{x_1+x_2+y_1},
\]
where all entries are calculated in \(\mathbb{F}_2\).

We verify the biquandle axioms. Since the two operations coincide,
\[
x\mathbin{\underline{\triangleright}}x
=
x\circ x
=
x\mathbin{\overline{\triangleright}}x
\]
for every \(x\in X\), so the diagonal axiom holds.

For each fixed \(y\in X\), both right translations are given by
\[
x\longmapsto Ax+By.
\]
Since
\[
\det(A)=1\in\mathbb{F}_2,
\]
these maps are bijective.

The switch map is
\[
S(x,y)
=
\left(
y\mathbin{\overline{\triangleright}}x,\,
x\mathbin{\underline{\triangleright}}y
\right)
=
(Ay+Bx,\,Ax+By).
\]
With respect to the ordered pair of column vectors \(\binom{x}{y}\),
the switch is represented by the block matrix
\[
M=
\begin{pmatrix}
B&A\\
A&B
\end{pmatrix}
=
\begin{pmatrix}
0&1&1&0\\
1&0&1&1\\
1&0&0&1\\
1&1&1&0
\end{pmatrix}.
\]
Since
\[
\det(M)=1\in\mathbb{F}_2,
\]
the switch map is bijective.

Finally, because the two biquandle operations coincide, all three
exchange laws reduce to the single identity
\[
(x\circ y)\circ(z\circ y)
=
(x\circ z)\circ(y\circ z).
\]
The matrices \(A\) and \(B\) satisfy
\[
A^2=B^2=I
\qquad\text{and}\qquad
AB+B^2=BA.
\]
Therefore,
\begin{align*}
(x\circ y)\circ(z\circ y)
&=A(Ax+By)+B(Az+By)\\
&=A^2x+BAz+(AB+B^2)y\\
&=A^2x+BAz+BAy\\
&=A^2x+BAy+(AB+B^2)z\\
&=A(Ax+Bz)+B(Ay+Bz)\\
&=(x\circ z)\circ(y\circ z).
\end{align*}
Thus the diagonal axiom, the required bijectivity axioms, and all
three exchange laws are satisfied. Hence \(X\) is a biquandle.
\end{proof}

\begin{lem}[Type and zero-flow vertices]\label{lem:t4-zero-flow}
The biquandle $T_4$ has parallel type two.  Its $\Z_2$-family is
\[
 x\ustar^{\,0}y=x\ostar^{\,0}y=x,\qquad
 x\ustar^{\,1}y=x\ostar^{\,1}y=x\circ y.
\]
For an ordered merge with incoming flows $a,b$ and outgoing flow
$a+b$, the colors are $(x,\kappa_a(x),x)$, where
$\kappa_0=\operatorname{id}$ and $\kappa_1=w$.  In particular,
\[
\begin{array}{c|c}
\text{ordered flows}&\text{ordered colors}\\ \hline
(1,1,0)&(x,w(x),x)\\
(1,0,1)&(x,w(x),x)\\
(0,1,1)&(x,x,x)
\end{array}
\]
Splits use these relations in reverse.
\end{lem}

\begin{proof}
With $A,B$ as above,
\[
 (x\circ y)\circ w(y)
 =A^2x+(AB+BA+B^2)y=x.
\]
Thus both second parallel operations are trivial; the first
operations are not trivial, so the type is exactly two.  The
$\Z_2$-family assertion is \cite[Theorem~4.5]{ishiiNelson}.
The vertex formulas are the partial multiplication in
\cite[Proposition~5.4]{ishiiNelson}; coloring invariance is
\cite[Proposition~5.2 and Corollary~5.3]{ishiiNelson}.
This also shows why a zero-flow edge cannot simply be assigned
an arbitrary color: it still participates in the vertex relation.
\end{proof}

\begin{exmp}[The $a,e$ attachment]\label{ex:t4-ae}
Use the semiarc labels of Nelson--Vo's Kishino diagram
\cite[Example~6]{nelsonVo2006}.  Add an edge $s$ at $a$ and $e$,
give the original edges flow $1$, and give $s$ flow $0$.  Choose
the ordered merge at $a$ with flows $(1,0,1)$ and the reversed
$(0,1,1)$ vertex at $e$, so that $s$ is directed from the latter
vertex to the former.  If $z$ is its color, the two vertex
relations read
\[
                            z=w(a),\qquad z=e.
\]
Crossings of $s$ with the old diagram may be taken virtual.
Thus restriction to Kishino colorings is bijective onto the
colorings satisfying $e=w(a)$.

For completeness, the four equations in the first half are
\[
 b\circ c=a,\qquad c\circ b=d,\qquad
 c\circ d=b,\qquad d\circ c=e.
\]
Their solutions are the following eight rows, obtained from the
$4\times4$ operation table:
\[
\begin{array}{c|ccc|c}
 a&b&c&d&e\\ \hline
 1&1&1&1&1\\
 1&2&3&2&1\\
 2&2&1&3&4\\
 2&1&3&4&4\\
 3&4&1&4&3\\
 3&3&3&3&3\\
 4&3&1&2&2\\
 4&4&3&1&2
\end{array}
\]
To check completeness, choose $b,c$; the first, second, and fourth
equations determine $a,d,e$, and the third tests the choice.
Exactly the eight displayed choices pass.  The other half has the
same equations with $(a,b,c,d,e)$ replaced by $(e,f,g,h,a)$.
Each half has two solutions for each incoming endpoint color and
sends that color through $w$.  Since $w^2=\operatorname{id}$,
the halves glue in $4\cdot2\cdot2=16$ ways.  The endpoint pairs are
\[
 (a,e)=(1,1),(2,4),(3,3),(4,2),
\]
each occurring four times.  All sixteen therefore extend uniquely across
this attachment.  This calculation specifies both the attachment
sites and the local vertex orders.  The ravel family in
Section~\ref{subsec:virtual-ravels} is constructed separately by
cutting unknot summands at $g$ and joining their $g$-arcs; its
coloring growth is proved directly in
Lemma~\ref{lem:ravel-coloring-growth}.
\end{exmp}

\subsection{A virtual type-three modification}

Let $T$ be the tetrahedral quandle.  A $\nu(3)$-move is the virtual
local replacement shown in Figure~\ref{fig:nu3}; an occurrence in a
graph diagram is shown in Figure~\ref{fig:nu3-action}.

\begin{prop}\label{prop:nu3}
Orient both strands of Figure~\ref{fig:nu3} upward.  The displayed
replacement gives a bijection of $T$-colorings fixing all exterior
colors.  A sequence of these replacements therefore preserves the
coloring number of any diagram in which they occur.
\end{prop}

\begin{proof}
Starting with endpoint colors $x,y$, propagation through the three
successive crossings changes $x$ to
\[
                  ((x\triangleright y)\triangleright y)
                    \triangleright y=x.
\]
The endpoint colors therefore agree on the two sides of the move.
Propagation is reversible, so this correspondence is a bijection on
colorings.  Composing these bijections proves the proposition.
\end{proof}

\begin{figure}[htbp]
\labellist
\small\hair 2pt
\pinlabel $x$ at -14 6
\pinlabel $y$ at 116 6
\pinlabel $x$ at -14 356
\pinlabel $y$ at 116 356
\pinlabel $x$ at 184 356
\pinlabel $y$ at 296 356
\pinlabel $x\triangleright y$ at 316 296
\pinlabel $y$ at 283 246
\pinlabel $(x\triangleright y)\triangleright y$ at 376 186
\pinlabel $((x\triangleright y)\triangleright y)\triangleright y$ at 428 51
\endlabellist
  \centering
  \includegraphics[width=.2\linewidth]{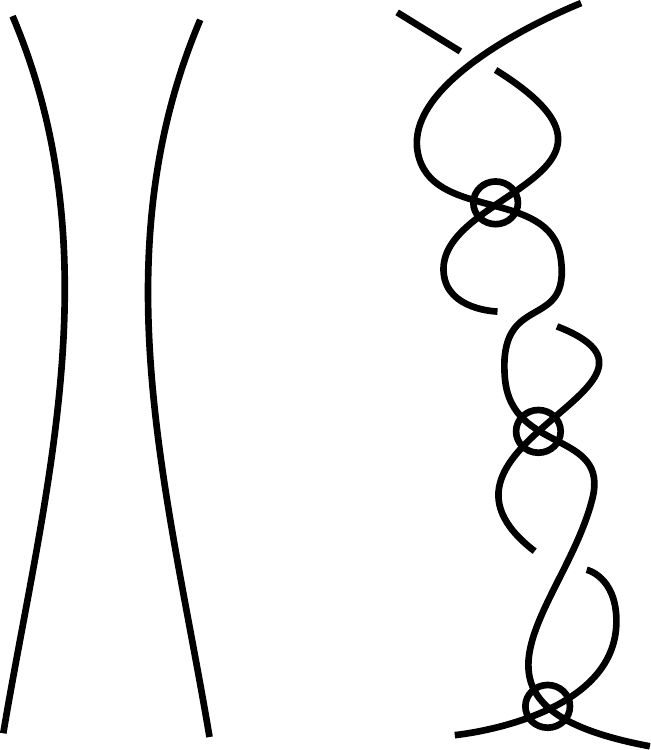}
  \caption{Color propagation through a $\nu(3)$-move.  Both strands
  are oriented from bottom to top.}
  \label{fig:nu3}
\end{figure}

\begin{figure}[htbp]
  \centering
  \includegraphics[width=.6\linewidth]{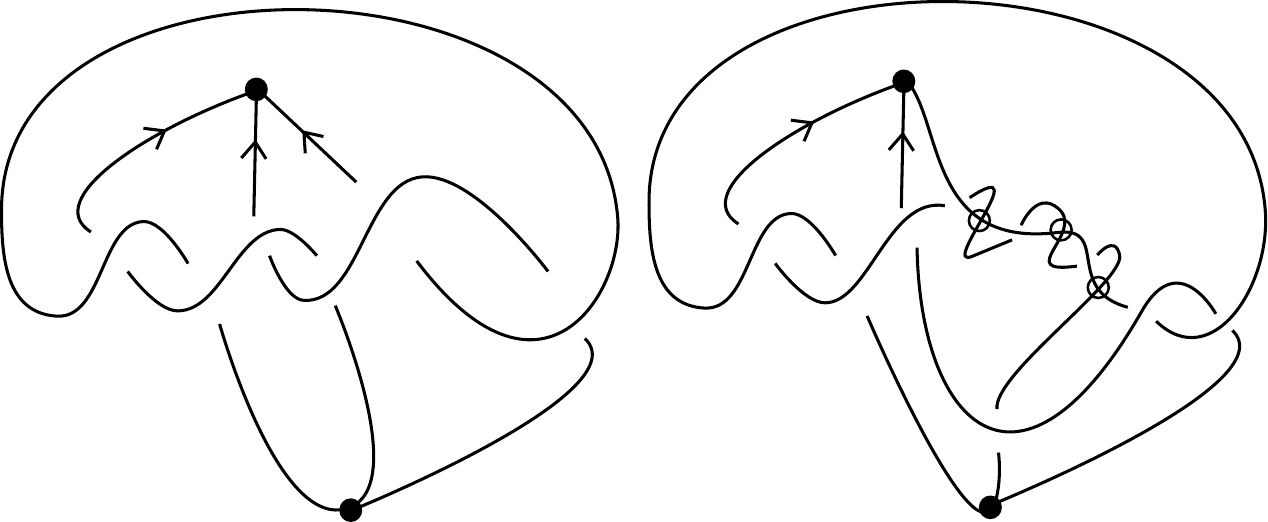}
  \caption{A $\nu(3)$-move in a graph diagram.}
  \label{fig:nu3-action}
\end{figure}

In the flowed virtual-biquandle framework, the flip virtual switch
recovers the ordinary virtual crossing rule used above.  A more
general virtual switch must satisfy
Definition~\ref{def:virtual-family} before its coloring count can be
used in a handlebody bridge bound.

\section{A detailed linear coloring calculation}\label{sec:linear}

Consider the graph obtained from the handcuff graph $7_{29}$ by
doubling one edge, labeled as in Figure~\ref{fig:labeled}.  At a
crossing, an $R_3$ coloring satisfies
\[
                         2x-y-z=0,
\]
where $x$ is the over-color and $y,z$ are the under-colors.  At each
four-valent vertex we use the Ishii--Yasuhara relation
$a-b+c-d=0$ in cyclic order.  The nine resulting equations are
\begin{align*}
 x_1-x_{10}+x_{11}-x_2&=0,&
 2x_4-x_2-x_3&=0,\\
 2x_3-x_6-x_7&=0,&
 2x_6-x_1-x_3&=0,\\
 2x_1-x_4-x_9&=0,&
 2x_7-x_4-x_5&=0,\\
 -x_7+x_8+x_{11}-x_{10}&=0,&
 2x_5-x_8-x_9&=0,\\
 2x_9-x_5-x_6&=0.&&
\end{align*}

\begin{figure}[htbp]
  \centering
  \includegraphics[width=.5\linewidth,trim=75 115 90 65,clip]{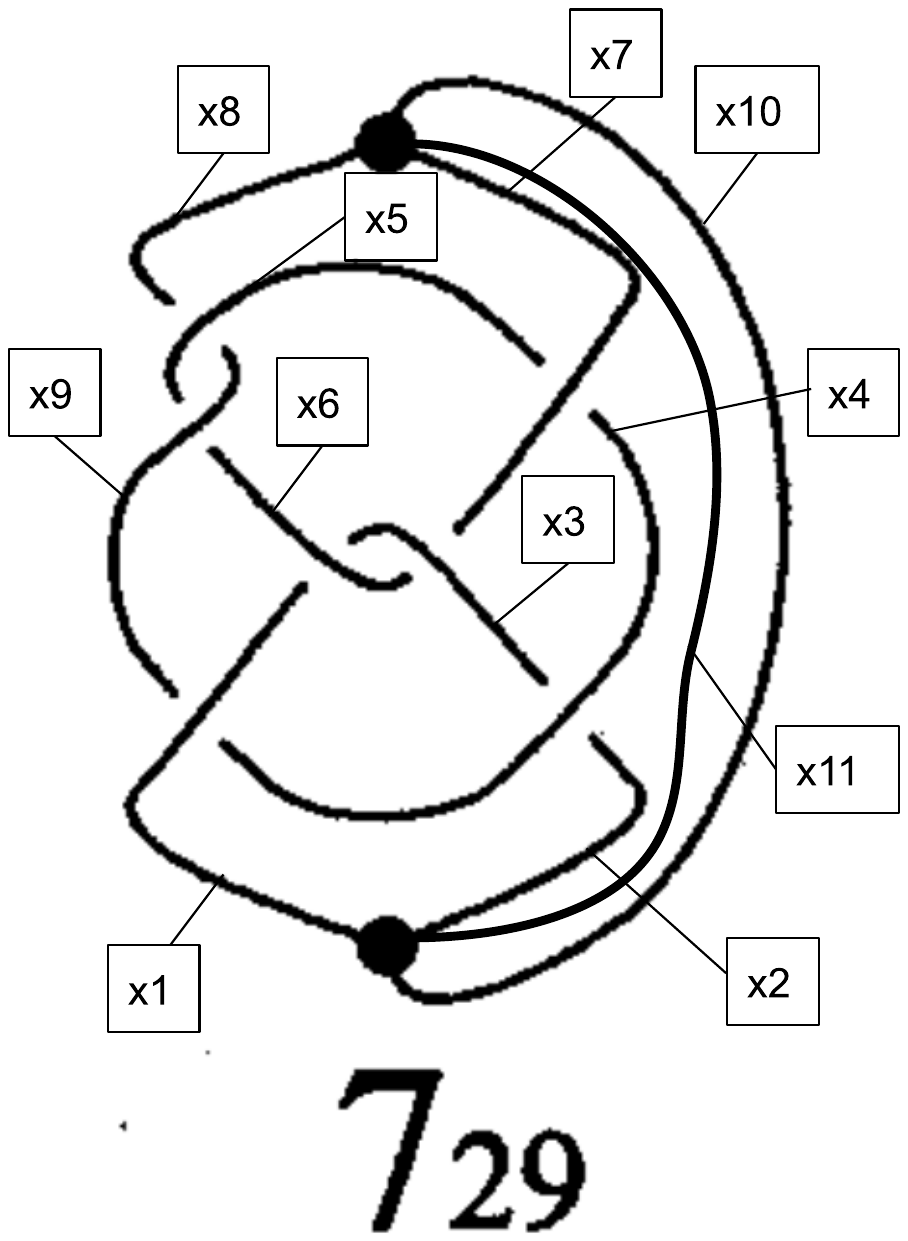}
  \caption{The labeled graph used in the $R_3$ calculation.}
  \label{fig:labeled}
\end{figure}

Row reduction over $\Z_3$ gives rank eight.  Taking
$x_9,x_{10},x_{11}$ as free variables, the dependent variables are
\begin{align*}
 x_1&=x_9+2x_{10}+x_{11},&
 x_2&=x_9+x_{10}+2x_{11},\\
 x_3&=x_9+x_{10}+2x_{11},&
 x_4&=x_9+x_{10}+2x_{11},\\
 x_5&=x_9,&x_6&=x_9,\\
 x_7&=x_9+2x_{10}+x_{11},&x_8&=x_9.
\end{align*}
Thus the coloring space has dimension three and the graph has
$3^3=27$ colorings.

This count illustrates the linear method.  It is not a
component-capacity estimate: the alternating vertex relation has
more local degrees of freedom than the equal-vertex relation.
The component-weight formula gives $b_c\geq6$, since the tangle
containing both four-valent vertices has weight at least
$2+2(4-2)$.  No upper bound is needed for this calculation.

\section{The overpass bridge index}\label{sec:overpass}

We use tree overpass presentations, in parallel with the forest
condition in the definition of a bridge splitting.

\begin{defn}[Tree overpass presentation]\label{def:overpass-presentation}
A \emph{tree overpass presentation} of a virtual spatial graph is a
diagram $D$ together with pairwise disjoint connected subgraphs
$\mathcal O=\{O_1,\ldots,O_s\}$ satisfying the following conditions:
\begin{enumerate}
\item Each $O_i$ is a tree or a proper arc, and its boundary consists
of points in edge interiors, away from crossings and graph vertices.
\item Every classical overcrossing belongs to an $O_i$, and no
classical undercrossing belongs to any $O_i$.
\item The closure of $D\setminus\bigcup_i O_i$ is a forest of trees
and proper arcs, and every connected component of $D$ meets
$\bigcup_i O_i$.
\end{enumerate}
A virtual crossing neither joins strands nor cuts them: the forest
conditions concern the underlying graph.  Additional crossing-free
arcs may be selected among the $O_i$ when needed to break a cycle.
For example, a crossing-free circle is divided into one overpass arc
and one complementary arc.
\end{defn}

One may start by cutting a diagram at its classical undercrossings
and looking at the resulting pieces that contain overcrossings.
For graphs these maximal pieces can contain cycles, and the
complement can contain cycles as well.  The forest conditions in
Definition~\ref{def:overpass-presentation} make explicit the extra
requirement needed to treat overpasses as tree components.
Existence follows from a bridge splitting by the construction in
Proposition~\ref{prop:overpass-bound} below.

\begin{defn}[Unweighted and weighted overpass indices]
The \emph{unweighted overpass bridge index} is
\[
 \widehat b_o(G)=\min_{(D,\mathcal O)}|\mathcal O|,
\]
where the minimum ranges over tree overpass presentations of $G$.
For an overpass tree $O$, define
\[
 \omega(O)=|\partial O|
          =2+\sum_{v\in V(O)}(\deg_G(v)-2).
\]
Here $\partial O$ consists of its boundary ends, and $V(O)$ consists
of its original graph vertices.  The \emph{weighted overpass bridge
index}, abbreviated to \emph{overpass bridge index}, is
\[
 b_o(G)=\min_{(D,\mathcal O)}
                 \sum_{O\in\mathcal O}\omega(O).
\]
The unweighted and weighted minima are taken separately.
\end{defn}

The weight formula is Lemma~\ref{lem:weight}: an arc has weight two,
a $d$-valent pod has weight $d$, and a tree containing several
vertices has the corresponding sum of excess valences.
On knots and links the forest requirement is automatic once every
component has an overpass, and all overpasses are arcs.  Consequently
$\widehat b_o$ is the usual first bridge index
\cite{nakanishiSatoh}, and $b_o=2\widehat b_o$ in that case.
For graphs, vertex weights prevent this reduction in general.
Since every tree component has weight at least two, one always has
$2\widehat b_o(G)\leq b_o(G)$.

\begin{prop}\label{prop:overpass-bound}
For every virtual spatial graph $G$,
\[
 \widehat b_o(G)\leq\widehat b(G),\qquad b_o(G)\leq b(G).
\]
For classical spatial graphs, $b_o(G)=b(G)$.
\end{prop}

\begin{proof}
Take a bridge splitting with upper tangle
$T_+=P_1\sqcup\cdots\sqcup P_s$.  Undoing its braid displays $T_+$
as a crossingless forest.  Move each classical undercrossing toward
the lower forest using the dragging moves in
Figures~\ref{fig:dragging} and~\ref{fig:dragging-move}, reversed
as necessary.  This
stretches edge collars without changing the abstract trees or their
numbers of boundary ends.  The upper trees become overpass trees
$O_1,\ldots,O_s$, while their complement is the lower forest with
extended collars.  Thus this is a tree overpass presentation, and
\[
 |\mathcal O|=s,\qquad
 \sum_{i=1}^s\omega(O_i)
 =\sum_{i=1}^s\omega(P_i)=|D\cap\ell|.
\]
Choosing the side with fewer components and then minimizing proves
$\widehat b_o(G)\leq\widehat b(G)$.  Minimizing the total weight
proves $b_o(G)\leq b(G)$.

For the classical reverse inequality, take a tree overpass
presentation.  Neither the overpass forest nor its complementary
forest has self-crossings: at every crossing the overpassing strand
belongs to the first forest and the underpassing strand to the
second.  Lift the first forest to a plane above the projection plane
and the second to a plane below it, joining their boundary points by
short monotone collars.  This preserves all crossing information,
and the two forests are trivial tangles.  Their separating plane,
completed to a sphere, meets the graph in
$\sum_i|\partial O_i|=\sum_i\omega(O_i)$ points.  Hence
$b(G)\leq\sum_i\omega(O_i)$, and minimizing gives the equality.
This is the tree-presentation form of the classical overpass
comparison in \cite[Corollary~3.3]{blackwell2026}.
\end{proof}

\begin{figure}[htbp]
  \centering
  \includegraphics[width=.6\linewidth]{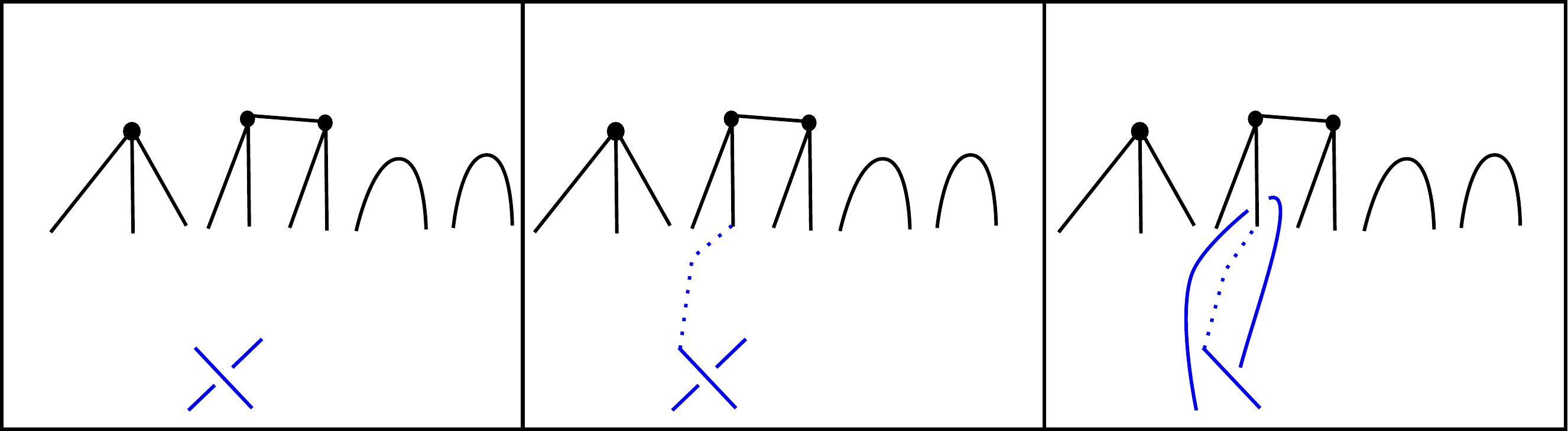}
  \caption{Dragging an undercrossing toward an upper tree component.}
  \label{fig:dragging}
\end{figure}

\begin{figure}[htbp]
  \centering
  \includegraphics[width=.6\linewidth]{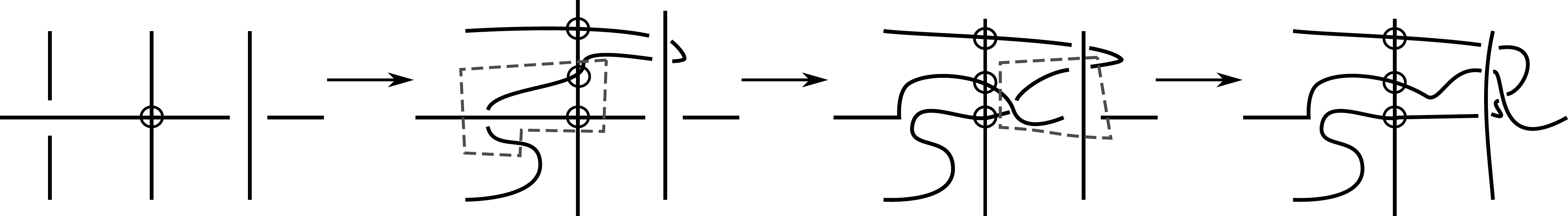}
  \caption{The Reidemeister moves used in the dragging operation.}
  \label{fig:dragging-move}
\end{figure}

\subsection{Virtual ravels with an unbounded overpass gap}
\label{subsec:virtual-ravels}

A virtual spatial graph is \emph{unknotted} if it is equivalent to a
crossingless planar diagram.  We call a virtual embedding of an
abstractly planar graph a \emph{virtual ravel} if it is not unknotted
but every subgraph homeomorphic to a disjoint union of circles is a
virtual unlink.  This is the virtual analogue of the usual ravel
condition \cite[Definition~1.2]{flapanMiller2017ravels}.  For a
bouquet, every circle subgraph is a single petal, and two circle
subgraphs cannot be vertex-disjoint.  Thus a nontrivial virtual
bouquet is a ravel precisely when each constituent petal is a
virtual unknot.

We first specify the local diagram and both stages of the
construction.  In Nelson--Vo's Kishino diagram
\cite[Example~6]{nelsonVo2006}, retain the right-hand two-crossing
half with semiarcs $e,f,g,h$, and join its $a$- and $e$-ends without
introducing classical crossings.  Denote the resulting diagram by
$U$.  It is a virtual unknot: after a virtual detour its two
oppositely signed classical crossings cancel by Reidemeister II.
The semiarc $g$ is the segment between the two overcrossings.
Remove a small open interval from $g$, obtaining an oriented
one-string virtual tangle $L$.  Write $g_-$ and $g_+$ for its
incoming and outgoing boundary semiarcs, respectively.  They are
distinct semiarcs until a boundary identification is imposed.

Take $n$ copies $L_1,\ldots,L_n$, and first join the outgoing end
of $L_i$ to the incoming end of $L_{i+1}$, with indices modulo $n$.
Let $c_i$ be the joining arc.  The result is one virtual knot
diagram $K_n$; these are the prescribed diagrammatic connected
sums along $g$.  All joining arcs may be routed with only virtual
crossings.  Next add an edge $s_i$ between $c_i$ and $c_{i+1}$ for
$1\leq i<n$.  On each intermediate arc $c_i$, use two distinct
attachment points, and let $j_i$ be the short interval between them.
Choose the attachments so that
\[
 E_n=s_1\cup\cdots\cup s_{n-1}\cup j_2\cup\cdots\cup j_{n-1}
\]
is a planar tree in a disk, with no classical crossings in that
disk.  Such a placement is obtained by arranging the joining arcs
along the boundary of a strip and connecting successive arcs
inside it; any additional routing intersections are virtual.
The resulting graph $\widetilde V_n$ is trivalent, and $E_n$
contains every graph vertex.  Contract $E_n$ to a single vertex
and call the resulting virtual spatial graph $V_n$.

The graph $K_n$ is a circle, and adjoining $n-1$ edges increases its
first Betti number from $1$ to $n$.  Contracting a tree preserves
this number.  Therefore $V_n$ is an $n$-petal bouquet, with one
vertex of valence $2n$.  Its petals are the tangles $L_i$ closed
through the contracted tree.  The case $n=3$ before contraction is
shown schematically in Figure~\ref{fig:ravel-expansion}.

\begin{figure}[htbp]
\centering
\begin{tikzpicture}[x=1.05cm,y=.95cm,line width=.65pt,
                    every node/.style={font=\small}]
  \foreach \x/\i in {0/1,3/2,6/3}{
    \draw (\x,-.36) rectangle (\x+1.3,.36);
    \node at (\x+.65,0) {$L_{\i}$};
  }
  \draw[->] (1.3,0)--(2.55,0);
  \draw (2.55,0)--(3,0);
  \draw[->] (4.3,0)--(5.7,0);
  \draw (5.7,0)--(6,0);
  \draw (7.3,0)--(7.8,0)--(7.8,-1.4)--(-.5,-1.4)--(-.5,0)--(0,0);
  \draw[->] (4,-1.4)--(3.5,-1.4);
  \draw[dashed,line width=1pt] (2.05,0)--(2.05,1.2)--(4.8,1.2)--(4.8,0);
  \draw[dashed,line width=1pt] (5.25,0)--(5.25,.85)--(8.45,.85)
                 --(8.45,-.7)--(7.8,-.7);
  \draw[line width=1.4pt] (4.8,0)--(5.25,0);
  \foreach \x/\y in {2.05/0,4.8/0,5.25/0,7.8/-.7}
    \fill (\x,\y) circle (2pt);
  \node[above] at (3.4,1.2) {$s_1\ (0)$};
  \node[above] at (6.85,.85) {$s_2\ (0)$};
  \node[below] at (2.05,-.02) {$c_1\ (1)$};
  \node[below] at (5.025,-.02) {$c_2\ (1)$};
  \node[below] at (3.8,-1.4) {$c_3\ (1)$};
\end{tikzpicture}
\caption{The trivalent expansion $\widetilde V_3$.  Each box is the
specified $g$-cut tangle $L$, including its classical and virtual
crossings.  Solid edges have flow $1$ and dashed edges have flow $0$.
The thick solid interval is $j_2$.  The tree
$E_3=s_1\cup j_2\cup s_2$ contracts to the bouquet vertex.
The schematic suppresses the crossings inside the boxes.}
\label{fig:ravel-expansion}
\end{figure}
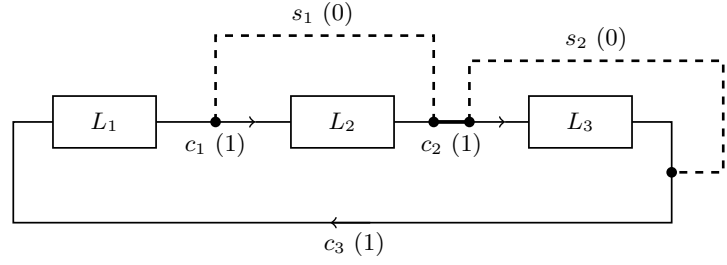

Orient $K_n$ coherently and orient each $s_i$ arbitrarily.  Give
every edge of $K_n$ flow $1$ and every $s_i$ flow $0$.  This defines
a $\Z_2$-flow $\rho_n$ on the Y-oriented graph $\widetilde V_n$:
each vertex has two incident flow-$1$ edges and one flow-$0$ edge.
Notice that the short intervals $j_i$ retain flow $1$; they belong
to the contracted tree but were not newly adjoined edges.  We use
the $\Z_2$-family of $T_4$ and the flip at virtual crossings.

\begin{lem}[Independent colorings of the summands]
\label{lem:ravel-coloring-growth}
For the fixed flow $\rho_n$ above,
\begin{equation}\label{eq:ravel-coloring-growth}
 \#\operatorname{Col}_{T_4}(\widetilde V_n,\rho_n)\geq 2^{n+1}.
\end{equation}
\end{lem}

\begin{proof}
In the modern sideways convention of
Section~\ref{subsec:kishino-graph}, the crossing equations of $L$
are
\begin{equation}\label{eq:ravel-block-relations}
 g_-\circ f=h,\qquad f\circ g_-=e,\qquad
 g_+\circ h=f,\qquad h\circ g_+=e.
\end{equation}
Here the two ends of the closing $e$-arc have already been
identified, while the two $g$-ends remain separate.  Substitution
in the displayed table for $\circ$ verifies each row below:
\begin{equation}\label{eq:ravel-block-colorings}
\begin{array}{c|ccc}
 g_-=g_+ & e&f&h\\ \hline
 1&1&1&1\\
 1&3&4&4\\ \hline
 3&1&2&2\\
 3&3&3&3
\end{array}
\end{equation}
We only need these four colorings, so no claim of completeness of
the tangle-coloring list is required.

Fix $c\in\{1,3\}$ and give every joining arc $c_i$ color $c$.
For each $L_i$, choose independently either of the two rows with
$g_-=g_+=c$.  All joining equations hold, so these choices give
$2^n$ distinct colorings of $K_n$.  Now give every added edge $s_i$
color $c$.  Since $w(c)=c$, all three underlying colors at every
vertex are $c$, and every ordered vertex relation in
Lemma~\ref{lem:t4-zero-flow} holds.  Thus all $2^n$ colorings
extend across the added edges.  The two choices of $c$ give
disjoint sets of colorings.  Their flow is the same $\rho_n$ in
both cases, proving \eqref{eq:ravel-coloring-growth}.
\end{proof}

\begin{proof}[Proof of Theorem~\ref{thm:kishino-overpass}]
We first determine the weighted overpass index.  In $V_n$, take
the central vertex together with the two initial portions of each
petal that contain its two overcrossings, stopping before the
undercrossings.  They form one overpass tree, a $2n$-valent pod.
The complement consists of $n$ proper arcs.  This gives
\[
 \widehat b_o(V_n)=1,\qquad b_o(V_n)\leq 2n.
\]
Conversely, in any tree overpass presentation of an $n$-bouquet,
the forest containing its unique vertex has at least one
component.  The component-weight formula shows that the number
of boundary ends of that forest is at least $2+(2n-2)=2n$.
Both complementary forests have the same boundary, so the total
overpass weight is at least $2n$, regardless of which forest
contains the vertex.  Hence
\begin{equation}\label{eq:ravel-overpass}
                         b_o(V_n)=2n.
\end{equation}

Now let $D=T_+\cup_\ell T_-$ be any bridge splitting of $V_n$,
and put $m=|D\cap\ell|$.  Suppose its unique vertex belongs to
$T_+$, and write $r=|\pi_0(T_+)|$ and $s=|\pi_0(T_-)|$.
Lemma~\ref{lem:weight} gives
\[
                   m=2r+2n-2=2s.
\]
In particular $s=r+n-1$, so $r$ is the smaller component count.
Expanding the vertex inside $T_+$ gives a trivalent representative
of $H_n=[V_n]_{\mathrm{hb}}$ with the same bridge splitting and
the same component counts.  Therefore
$\widehat b(H_n)\leq r$.  On the other hand, the fixed-flow bridge
bound and Lemma~\ref{lem:ravel-coloring-growth} give
\[
 2^{n+1}\leq
 \#\operatorname{Col}_{T_4}(\widetilde V_n,\rho_n)
 \leq 4^{\widehat b(H_n)}\leq 4^r.
\]
Thus $r\geq\lceil(n+1)/2\rceil$.  Since the splitting was
arbitrary,
\begin{equation}\label{eq:ravel-bridge-lower}
 b(V_n)\geq 2n-2+2\left\lceil\frac{n+1}{2}\right\rceil.
\end{equation}
Combining \eqref{eq:ravel-overpass} and
\eqref{eq:ravel-bridge-lower} yields
\[
 b(V_n)-b_o(V_n)
 \geq 2\left\lceil\frac{n+1}{2}\right\rceil-2
 \geq n-1,
\]
which tends to infinity.

It remains to verify the ravel condition for $V_n$.  After deleting
all other petals and suppressing the resulting bivalent vertex,
any one petal is exactly the closure of $L$, up to virtual detours.
It is therefore the virtual unknot $U$.  There are no subgraphs
consisting of two or more disjoint circles, because all petals
share the unique vertex.  Finally, $V_n$ is nontrivial for
$n\geq2$: its bridge index is strictly greater than $2n$, whereas
the planar $n$-bouquet has a bridge splitting into a $2n$-valent
pod and $n$ trivial arcs, and consequently has bridge index $2n$.
This proves that $V_n$ is a virtual ravel and completes the proof.
\end{proof}

The ravel assertion concerns the bouquets $V_n$.  The trivalent
expansions $\widetilde V_n$ are used to apply the fixed-flow
invariant to their common handlebody classes.  A ravel condition
does not require arbitrary proper graph subgraphs to be unknotted;
the next section distinguishes this stronger condition from the
ravel property in the classical constructions.

\section{Classical ravels and prescribed Euler characteristic}
\label{sec:correction}

A classical embedding of an abstractly planar graph is a \emph{ravel}
if it is nonplanar and every subgraph homeomorphic to a disjoint union
of circles is an unlink.  It is \emph{almost unknotted} if it is
nonplanar and every proper subgraph is planar.  The latter condition
implies the former.  For a two-petal bouquet the conditions coincide,
but for bouquets with more petals they need not coincide.

The main construction in this section is compatible tetrahedral
clasping.  We insert its three-strand version near a vertex and
iterate it in disjoint balls.  The fixed boundary colors give an
exponential coloring lower bound.  A constituent link omits at
least one of the three participating half-edges, which is enough
to undo the clasp.  Thus the same construction works for both
trivalent graphs and bouquets.

\subsection{A lower bound for almost unknotted graphs}

\begin{lem}\label{lem:two-components}
Let $\Gamma\subset S^3$ be a nontrivial almost unknotted embedding of
a connected planar graph with $\chi(\Gamma)<0$.  In any bridge
splitting, both trivial tangles have at least two components.
\end{lem}

\begin{proof}
Suppose the upper tangle is a single trivial tree.  Contract its
internal edges to obtain a pod, preserving the regular neighborhood
of the graph.  A trivial pod is equivalent to radial segments from
the center of a ball to its boundary leaves.  A homeomorphism of the
boundary sphere extends over this pod by coning.  Consequently the
boundary braiding of the lower trivial forest can be absorbed at the
pod, and the resulting spine is planar.  Thus $N(\Gamma)$ is an
unknotted handlebody.

Its exterior is also a handlebody, with compressible boundary of
genus $1-\chi(\Gamma)\geq2$.  This contradicts the theorem of
Ozawa--Tsutsumi \cite{OzawaTsutsumi2003}: a nontrivial minimally
knotted embedding of a planar graph has exterior with incompressible
boundary.  Interchanging the two sides proves the assertion for the
lower tangle as well.  The contraction in this argument is a change
of spine; it is used to identify the regular neighborhood.
\end{proof}

\begin{cor}\label{cor:lowerboundarygvb}
A connected almost unknotted classical embedding $\Gamma$ of a
planar graph with $\chi(\Gamma)=-x<0$ satisfies
\[
                         b(\Gamma)\geq x+4.
\]
\end{cor}

\begin{proof}
If the two tangles have $c_+$ and $c_-$ components and the bridge
sphere meets the graph in $q$ points, then
$-x=c_++c_--q$.  Lemma~\ref{lem:two-components} gives
$q=x+c_++c_-\geq x+4$.
\end{proof}

\subsection{Compatible tetrahedral clasping}
\label{subsec:tetrahedral-clasping}

Let $C$ be the oriented three-strand tangle in the upper-left
panel of Figure~\ref{fig:tetrahedral-clasping}.  A \emph{tetrahedral
clasp replacement} replaces a trivial three-strand tangle by $C$,
matching its six endpoints, their pairing, and their orientations.
The replacement is supported in a ball and changes no vertices or
edges of the abstract graph.  The other panels show longer versions
of the compatible clasping pattern.  The constructions below use
copies of the three-strand panel.

\begin{figure}[!htbp]
  \centering
  \includegraphics[width=.98\linewidth]{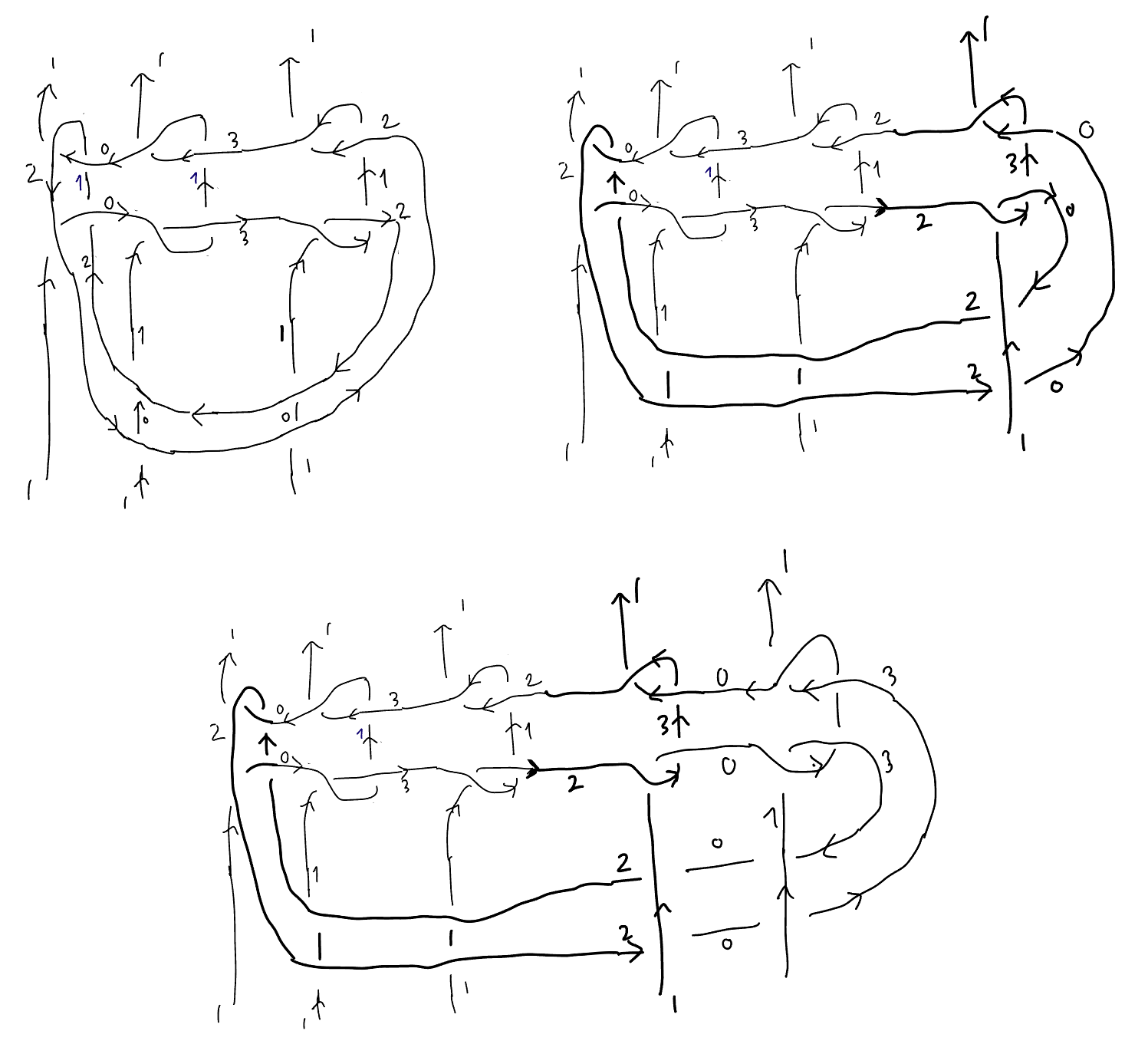}
  \caption{Compatible tetrahedral clasping, with its nonconstant
  colorings.  The labels $0,1,2,3$ are quandle colors, and all
  boundary endpoints have color $1$.  The upper-left three-strand
  tangle is the replacement $C$ used in the proofs.  The remaining
  panels display longer clasping patterns.  These color labels
  are distinct from the $\Z_3$ flow values used below.}
  \label{fig:tetrahedral-clasping}
\end{figure}

\begin{lem}[Compatible tetrahedral clasping]
\label{lem:tetrahedral-clasping}
The tangle $C$ has the following properties.
\begin{enumerate}
\item Deleting any one of its three strands leaves the trivial
tangle with the same surviving endpoints and pairing, relative
to the boundary of the replacement ball.
\item For each $t\in T$, there are at least four $T$-colorings of
$C$ assigning $t$ to all six boundary endpoints.
\end{enumerate}
\end{lem}

\begin{proof}
For (i), distinguish the strand containing the long return arcs
from the other two strands.  If the return strand is deleted, the
remaining two strands straighten independently.  If either of the
other strands is deleted, its place provides a gap through which
the return arcs can be pulled.  At the surviving clasp, the two
opposite crossings then cancel by a Reidemeister II move; the
opened return straightens, and any remaining kink is removed by
a Reidemeister I move.  These moves take place inside the ball
and fix its boundary.  This is the three-strand clasp-deletion
argument; compare \cite[Section~3.2]{jang2016brunnian}.

For (ii), first take $t=1$.  To check the displayed nonconstant
coloring, the right translations of the table of $T$ are
\[
 S_0=(1\ 3\ 2),\quad S_1=(0\ 2\ 3),\quad
 S_2=(0\ 3\ 1),\quad S_3=(0\ 1\ 2).
\]
At a crossing use the indicated right translation or its inverse,
according to the normal direction of the overpassing strand.
In particular, on the successive horizontal portions under a
color-$1$ strand, the displayed labels change by $S_1$ in one
direction and by $S_1^{-1}$ on the return.  The translations
$S_0,S_2,S_3$ check the closing crossings.  The six exterior
labels are all $1$, as drawn.

Every right translation is a quandle automorphism.  Apply
$\operatorname{id},S_1,S_1^2$ to the displayed coloring.  Their
effect on the labels is
\[
\begin{array}{c|rrrr}
 &0&1&2&3\\ \hline
\operatorname{id}&0&1&2&3\\
S_1&2&1&3&0\\
S_1^2&3&1&0&2
\end{array}
\]
so all boundary labels stay $1$.  The three interior colorings
are distinct, since the displayed coloring uses a label different
from $1$.  The constant coloring with value $1$ is a fourth
coloring.  Finally, the inner automorphism group acts transitively
on $T$: the translations above carry $1$ to each of $0,2,3$.
An automorphism carrying $1$ to $t$ transports these four
colorings to four colorings with boundary color $t$.
\end{proof}

\begin{lem}[Independent clasp choices]\label{lem:clasp-iteration}
Let $D_0$ be a connected planar graph diagram with an orientation
for which the equal-vertex $T$-coloring rule is invariant.  At one
vertex choose three distinct incident half-edges.  Insert $k\geq1$
copies of $C$ in disjoint balls, each meeting those same half-edges
once in a sufficiently small vertex collar.  Write $D_k$ for the
result.  Then
\[
                 \#\operatorname{Col}_T(D_k)\geq4^{k+1}.
\]
Every subgraph of $D_k$ homeomorphic to a disjoint union of circles
is an unlink.  In particular, $D_k$ is a ravel.
\end{lem}

\begin{proof}
The three chosen half-edges are disjoint arcs in a ball about the
vertex.  Prepare $k$ disjoint short collars along them, each a
ball containing a trivial three-strand tangle.  For each finite
$k$ these collars fit between the vertex and the ends of the
chosen small half-edges.  Identify each tangle, with its
orientations, with the trivial tangle replaced by $C$.

Fix $t\in T$ and assign $t$ to every portion of $D_k$ outside
the replacement balls.  At the boundary of each ball all six
labels are consequently $t$.  Lemma~\ref{lem:tetrahedral-clasping}
gives four choices inside each ball.  They agree with the fixed
outside coloring and can be chosen independently.  Different
choices are distinguished by restriction to the ball in which
they differ.  This gives $4^k$ colorings for each $t$; varying
$t$ gives $4^{k+1}$, since the color at the original vertex
distinguishes the four sets.  This argument counts a specified
subset of colorings and does not require an exact count.

Now let $L$ be a subgraph homeomorphic to a disjoint union of
circles.  At the chosen vertex, $L$ uses either zero or two
incident half-edges.  It therefore omits at least one of the three
selected half-edges in every replacement ball.  Part (i) of
Lemma~\ref{lem:tetrahedral-clasping} undoes the replacement on
the surviving strands, fixing the boundary of that ball.  If more
strands are absent, restrict the same isotopy to those that remain.
The balls are disjoint, so all these isotopies can be performed
without interference.  The resulting link is a subgraph of the
planar diagram $D_0$, and hence is an unlink.

A connected planar diagram has exactly four equal-vertex
$T$-colorings, since all its edges must have the same color.
The invariant count above is greater than four, so $D_k$ is
nonplanar.  Together with the constituent-link assertion, this
proves the ravel property.
\end{proof}
\FloatBarrier

\subsection{Non-Eulerian ravels at every negative Euler characteristic}
\label{subsec:non-eulerian}

For $x\geq1$, let $\Lambda_x$ be the following connected planar
trivalent graph.  Set $\Lambda_1=\theta_3$.  For $x\geq2$,
start with a cycle on $2x$ vertices and double each edge in an
alternating perfect matching.  Thus $\Lambda_2$ is a square with
two opposite edges doubled, and $\Lambda_3$ is a hexagon with
three alternate edges doubled.  There are $2x$ vertices and
$3x$ edges, giving $\chi(\Lambda_x)=-x$.  All vertices have
odd valence, so these graphs are non-Eulerian.

\begin{prop}\label{thm:non-eulerian}
For every $x,k\geq1$, there is a classical ravel $G_{x,k}$ with
underlying abstract graph $\Lambda_x$ such that
\[
 \chi(G_{x,k})=-x,\qquad b(G_{x,k})\geq x+2k+2.
\]
For one fixed $\Z_3$-flow $\rho_{x,k}$ it has at least $4^{k+1}$
tetrahedral-family colorings.  When $x=1$, these examples are
almost unknotted theta-curves.
\end{prop}

\begin{proof}
The graph $\Lambda_x$ is bipartite.  Orient every edge from the
first part to the second; every vertex is then a source or a
sink, with signed valence $-3$ or $3$.  The equal-vertex
$T$-coloring rule is therefore invariant by
Proposition~\ref{prop:equal-vertex}.

Start with a planar embedding of this oriented graph, choose a
vertex, and apply $k$ tetrahedral clasp replacements to its three
incident half-edges as in Lemma~\ref{lem:clasp-iteration}.  The
result is $G_{x,k}$.  Its abstract graph and Euler characteristic
are unchanged.  That lemma proves the ravel property and gives
at least $4^{k+1}$ colorings.

To express the count using one fixed flow, choose a perfect
matching $M$ of $\Lambda_x$.  Reverse the edges of $M$ and give
them flow $2=-1$ in $\Z_3$; give every other edge flow $1$.
Each vertex is now Y-oriented, with the flow equation $1+1=2$.
For the parallel quandle family the crossing operations are
$u\ustar^{\,j}v=S_v^j(u)$ and $u\ostar^{\,j}v=u$.
Reversing an edge and negating its flow leaves the signed powers
at its crossings unchanged.  The diagonal maps are the identity,
so the vertex rule remains equality of the three colors.  Thus
the colorings above give at least $4^{k+1}$ colorings for this
single flow $\rho_{x,k}$.

In any bridge splitting, the fixed-flow propagation bound gives
\[
 4^{k+1}\leq4^{c_+},\qquad 4^{k+1}\leq4^{c_-},
\]
where $c_\pm$ are the two tree-component counts.  Hence
$c_+,c_-\geq k+1$, and additivity of Euler characteristic gives
\[
 |G_{x,k}\cap S|=x+c_++c_-\geq x+2k+2.
\]
Minimizing proves Theorem~\ref{thm:anyEuler}(i).  For $x=1$,
deleting any theta edge removes a participating strand in every
ball, so the same deletion isotopies give an unknot.  These
theta-curves are therefore almost unknotted.
\end{proof}

\subsection{Bouquets and even bridge indices}
\label{subsec:even-bouquets}

The initial bouquet construction in \cite{blackwell2026},
in the proof of Theorem~1.2, gives the following
sharp minimum.  We use this initial construction separately from
the subsequent three-strand replacement in that proof: a replacement
supported on only three petals does not establish almost
unknottedness after deletion of a petal outside its support.

\begin{prop}[The initial bouquet family]\label{prop:bouquet-family}
For every $r\geq2$, there is an almost unknotted classical
$r$-petal bouquet $B_r$ with $b(B_r)=2r+2$.  This is the minimum
bridge index among almost unknotted $r$-petal bouquets.
Every bridge splitting of an $r$-petal bouquet has an even number
of intersections.
\end{prop}

\begin{proof}
Let $s$ be the number of components on the side containing the
unique vertex.  The weight formula gives
\begin{equation}\label{eq:bouquet-weight}
                         |B\cap S|=2s+2r-2.
\end{equation}
This proves evenness.  For an almost unknotted bouquet,
Lemma~\ref{lem:two-components} gives $s\geq2$, hence
$b(B)\geq2r+2$.

For attainment, use the initial bouquet of
\cite[proof of Theorem~1.2, Figures~34--36]{blackwell2026},
reproduced in Figures~\ref{fig:twist-tangle} and~\ref{fig:bouquet}.
Its $r$ boxes are the three-string tangles from Fox
$3$-colorable twist knots, joined with the endpoint identifications
shown there.  The construction has a nonconstant equal-vertex
$R_3$ coloring, so it is nonplanar.  After any petal is deleted,
the vertex seed propagates through the first opened box and then
successively through the remaining boxes, as shown in
Figure~\ref{fig:twist-tangle}.  The resulting one-seed presentation
is planar by \cite[Lemma~5.23]{blackwell2026}.  This proves almost
unknottedness.  Before deletion, a vertex seed of weight $2r$ and
one arc seed of weight $2$ propagate through all boxes.
The Wirtinger--bridge correspondence
\cite[Theorem~1.1]{blackwell2026} therefore gives a bridge splitting
with $2r+2$ intersections.  Together with the lower bound, this
proves the claimed minimum.
\end{proof}

\begin{figure}[htbp]
\labellist
\small\hair 2pt
\pinlabel {$NW$} at 460 173
\pinlabel {$NC$} at 505 173
\pinlabel {$NE$} at 600 173
\pinlabel {$SW$} at 460 -13
\pinlabel {$SC$} at 505 -13
\pinlabel {$SE$} at 600 -13
\endlabellist
  \centering
  \includegraphics[width=.9\linewidth]{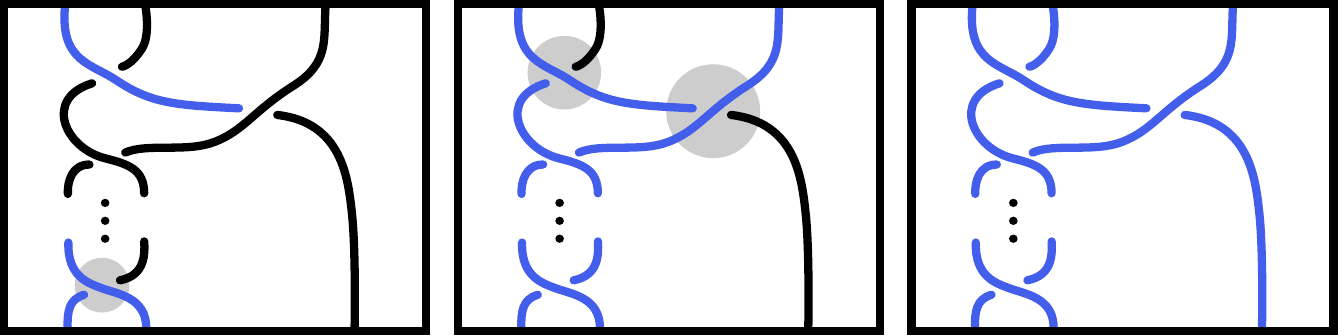}
  \caption{Seed propagation through the twist-knot tangle used
  in the initial bouquet construction.}
  \label{fig:twist-tangle}
\end{figure}

\begin{figure}[htbp]
  \centering
  \includegraphics[width=.7\linewidth]{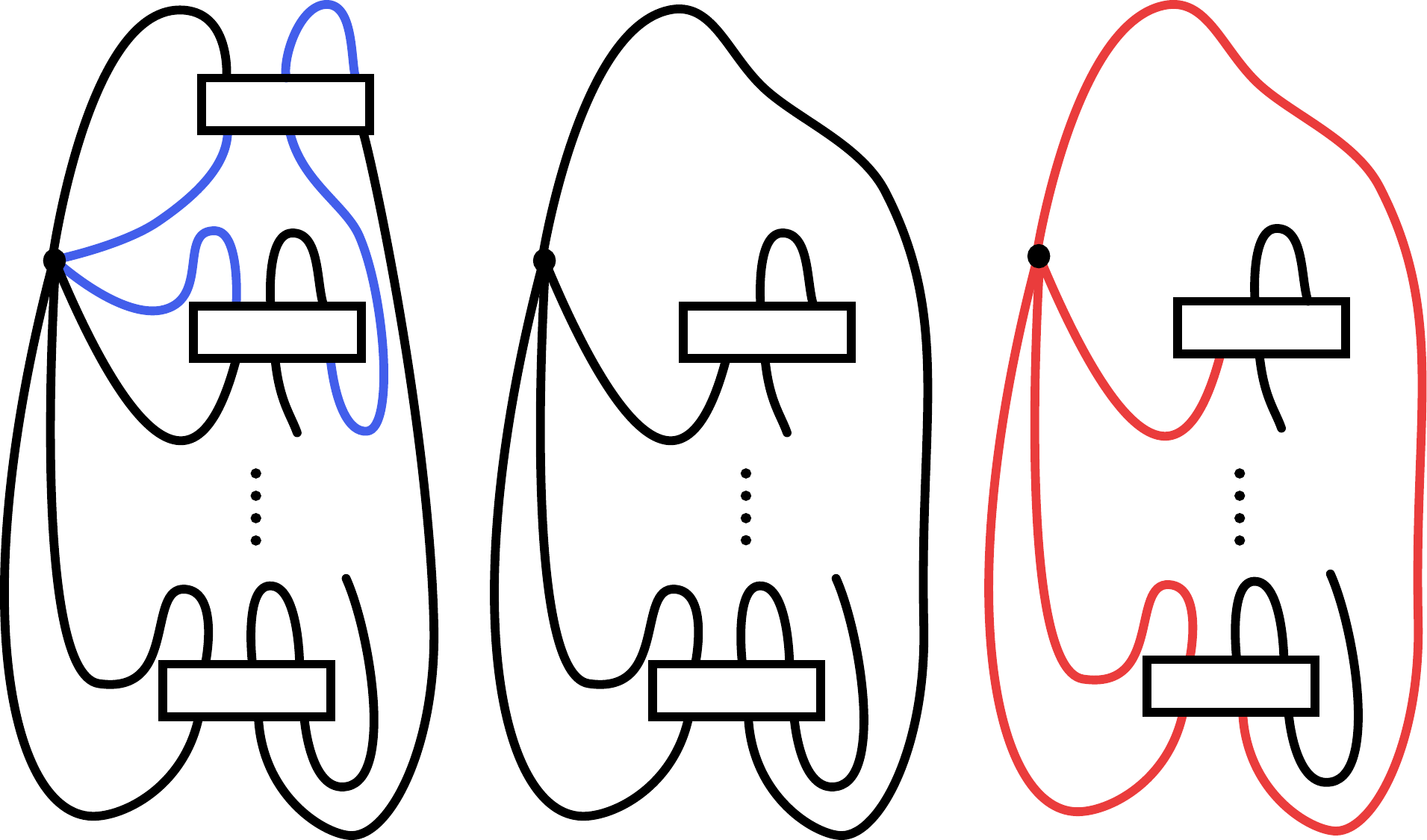}
  \caption{The initial almost unknotted bouquet $B_r$, deletion
  of a petal, and the start of vertex-seed propagation.}
  \label{fig:bouquet}
\end{figure}

The same tetrahedral clasping gives an unbounded family at each
fixed bouquet rank.  Orient each petal as a loop; the unique vertex
then has signed valence zero, so the equal-vertex rule for $T$ is
invariant for every rank.

\begin{prop}\label{prop:ravel-bouquets-fixed}
For every $r\geq2$ and $k\geq1$, there is a classical ravel
$r$-petal bouquet $C_{r,k}$ such that
\[
                         b(C_{r,k})\geq2k+2r.
\]
Its bridge index is even.  For $r=2$ the examples are almost
unknotted.
\end{prop}

\begin{proof}
Start with a planar oriented $r$-petal bouquet and choose three
distinct half-edges at its vertex.  Insert $k$ copies of the
three-strand tetrahedral clasp in disjoint collars of those
half-edges.  Lemma~\ref{lem:clasp-iteration} gives the resulting
ravel $C_{r,k}$ at least $4^{k+1}$ equal-vertex $T$-colorings.
Notice that a petal uses only two half-edges at the vertex, even
when both belong to the chosen set of three; it therefore always
omits a strand of every clasp.

In an arbitrary bridge splitting, let $s$ be the component count
on the side containing the vertex.  One color per tree determines
at most one coloring of the whole graph, by
Lemma~\ref{lem:trivial-coloring}.  Consequently
$4^{k+1}\leq4^s$, and $s\geq k+1$.
Formula~\eqref{eq:bouquet-weight} now gives
\[
 |C_{r,k}\cap S|=2s+2r-2\geq2k+2r.
\]
This proves the lower bound and evenness.  When $r=2$, the ravel
condition is the almost unknotted condition, as noted at the
start of this section.  The abstract bouquet is unchanged as
$k$ increases, proving Theorem~\ref{thm:anyEuler}(ii).
\end{proof}
\FloatBarrier

\subsection{Bouquets as spatial graphs and as handlebody-knots}
\label{sec:spine-comparison}

The next construction gives an explicit bridge gap.  Its upper
bound on handlebody bridge index comes from a specified trivalent
presentation.

\begin{prop}\label{prop:bouquet-spine-count}
For every $r\geq2$, there is a classical ravel $r$-petal bouquet
$P_r$ such that
\[
 b(P_r)=2r+2,\qquad b_{\mathrm{hb}}(N(P_r))\leq r+4.
\]
Consequently,
\[
 b(P_r)-b_{\mathrm{hb}}(N(P_r))\geq r-2
                    =-\chi(P_r)-1\longrightarrow\infty.
\]
\end{prop}

\begin{proof}
Take the fixed almost unknotted two-petal bouquet $B_2$ from
Proposition~\ref{prop:bouquet-family}, with $b(B_2)=6$.
Obtain $P_r$ by adjoining $r-2$ small planar petals at its vertex.
The old two petals and the new petals are all unknots; no two are
vertex-disjoint.  The subgraph $B_2$ proves nonplanarity, so $P_r$
is a ravel.

In the six-intersection bridge presentation of $B_2$, add each
new petal as two short legs from the vertex-containing tree and
one boundary-parallel arc in the opposite ball.  Each petal adds
two intersections, giving $b(P_r)\leq6+2(r-2)=2r+2$.
Conversely, if a bridge splitting of a bouquet has only one tree
on the vertex-containing side, that tree is a pod.  Flatten the
opposite trivial arc tangle and absorb its boundary braid at the
pod vertex, as in \cite[Lemma~5.23]{blackwell2026}.  The bouquet
would be planar.  Nonplanarity therefore forces at least two
components on the vertex-containing side.  Formula
\eqref{eq:bouquet-weight} gives $b(P_r)\geq2r+2$.

For the handlebody upper bound, choose a trivalent spine of the
already constructed neighborhood $N(P_r)$.  First expand the
four-valent vertex in the six-intersection presentation of $B_2$;
Lemma~\ref{lem:handlebody-tangles} preserves those intersections.
For each of the $r-2$ small planar petals, slide its two attaching
intervals along this spine and separate them onto opposite sides
of the bridge sphere on a terminal branch.  The petal is then
represented by two parallel segments between two trivalent
junctions.  The disk between the segments is the disk of the
original small planar petal.  This specifies a change of spine
inside $N(P_r)$; collapsing the connecting tree recovers its
bouquet spine.

In each bridge ball, the terminal branch has become a Y with
two boundary leaves.  It is still a trivial tree, and the other
tree components are unchanged.  Each separated petal contributes
one additional intersection, since two parallel segments replace
one segment across the sphere.  Choose the petal disks disjoint
so that these changes can be made simultaneously.  The resulting
trivalent spine has $6+(r-2)=r+4$ intersections and represents
$N(P_r)$.  Hence $b_{\mathrm{hb}}(N(P_r))\leq r+4$.
Subtracting proves the gap and Theorem~\ref{thm:anyEuler}(iii).
\end{proof}

\section{The constrained and unconstrained indices}\label{sec:gap}

\subsection{Colorings of trivial tangles}

The bridge estimate needed below uses the equal-vertex quandle
rule.  Its one-color-per-tree property is the reason that this
count bounds the number of components in either bridge tangle.

\begin{lem}\label{lem:trivial-coloring}
Let $T$ be a trivial tangle with $s$ connected components, colored by
a quandle $Q$ using the equal-vertex rule.  Then
\[
                         \#\operatorname{Col}_Q(T)\leq |Q|^s.
\]
If $T$ is either tangle in a bridge splitting of $G$, and the
coloring rule is invariant under the allowed graph moves, then
\[
                         \#\operatorname{Col}_Q(G)\leq |Q|^s.
\]
\end{lem}

\begin{proof}
Stack a braid that turns $T$ into a crossingless forest.  Stacking is
reversible, so it gives a bijection on coloring sets.  On a connected
crossingless tree or proper arc, the equal-vertex relation forces all
semi-arcs to have the same color.  There are therefore at most $|Q|$
choices per component and at most $|Q|^s$ choices for $T$.

For the second assertion, put the splitting into crossing-free
upper and lower forests separated by a braid.  Choose the forest
corresponding to $T$ as the starting side.  Its colors determine
all braid colors by the invertible crossing rules.  Every component
of the other forest has a boundary leaf, whose color determines
all colors in that tree by the equal-vertex rule.  Thus the starting
colors extend to at most one coloring of $G$, proving the bound.
\end{proof}

We will also use the following connected-sum estimate.  A finite
quandle $Q$ is \emph{homogeneous} if its automorphism group acts
transitively on $Q$.

\begin{lem}[Connected-sum lower bound]\label{lem:connected-sum}
Let $Q$ be a finite homogeneous quandle, and let $D_1\#D_2$ be formed
by removing small open intervals from crossing-free edge subarcs
and joining the two resulting pairs of endpoints without new crossings.
Then
\[
 \#\operatorname{Col}_Q(D_1\#D_2)
 \geq\frac{\#\operatorname{Col}_Q(D_1)\,
         \#\operatorname{Col}_Q(D_2)}{|Q|}.
\]
\end{lem}

\begin{proof}
For $x\in Q$, let $c_i(x)$ be the number of colorings of $D_i$ in
which the chosen crossing-free gluing subarc has color $x$.  Homogeneity supplies a bijection
between the colorings counted by $c_i(x)$ and $c_i(y)$, so
$c_i(x)=\#\operatorname{Col}_Q(D_i)/|Q|$.  A pair of summand colorings with the same color on the gluing
edge gives a coloring of the connected sum.  Distinct pairs give
distinct colorings, since restriction recovers the pair.  Therefore
\[
 \#\operatorname{Col}_Q(D_1\#D_2)
   \geq\sum_{x\in Q}c_1(x)c_2(x),
\]
which is the stated lower bound.
\end{proof}

For a $k$-strand vertex sum, remove the two vertex neighborhoods
first.  If $r_{\mathbf x}$ and $s_{\mathbf x}$ count colorings of
these \emph{punctured} diagrams with prescribed boundary tuple
$\mathbf x$ (ordered by the gluing), then
\[
 \#\operatorname{Col}(D_1\#_kD_2)
       =\sum_{\mathbf x}r_{\mathbf x}s_{\mathbf x}.
\]
The sum is over all matching boundary tuples.  Restricting to
constant tuples coming from colorings of the original closed
summands gives a lower bound, which is sufficient for our
constructions.

\begin{figure}[htbp]
  \centering
  \includegraphics[width=.3\linewidth]{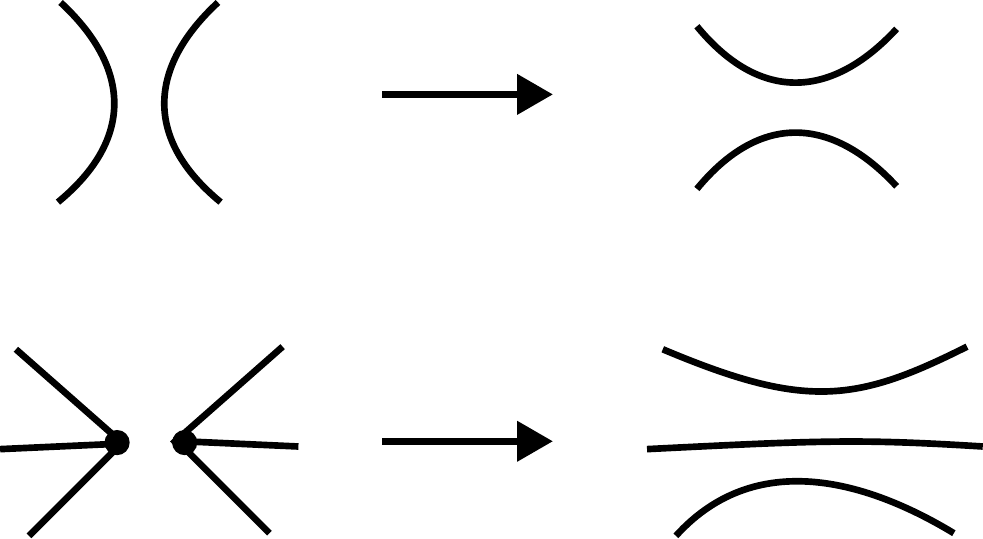}
  \caption{An edge connected sum (top) and a three-strand vertex sum
  (bottom).}
  \label{fig:sums}
\end{figure}

\subsection{An unbounded gap}

Let $S$ be the Suzuki $\theta_4$-curve, and let $G_n=\#^n S$ be the
$n$-fold edge connected sum.  Thus $G_n$ has $2n$ vertices, all of
valence four.

\begin{proof}[Proof of Theorem~\ref{thm:main-gap}]
For one summand the presentation has six bridge-sphere
intersections.  Form an edge connected sum at small balls centered
on bridge-sphere intersections in two summands.  Each such ball
meets the graph in one unknotted arc.  Gluing the punctured bridge
spheres joins one tree from each summand on each side; the joined
components are still trivial trees.  Two intersection points are
removed, so the counts add with a subtraction of two.  Iterating
gives the presentations in Figure~\ref{fig:unconstrained-sum} and
\begin{equation}\label{eq:unconstrained-upper}
                             b(G_n)\leq 4n+2.
\end{equation}

Use the equal-vertex coloring by $R_3$.  The Suzuki $\theta_4$-curve
has nine such colorings \cite[Example~5.8, Figure~20]{blackwell2026}.  Since
$R_3$ is homogeneous, Lemma~\ref{lem:connected-sum} gives
\begin{equation}\label{eq:suzuki-colorings}
 \#\operatorname{Col}_{R_3}(G_n)
   \geq\frac{9^n}{3^{n-1}}=3^{n+1}.
\end{equation}

Lemma~\ref{lem:trivial-coloring} applies to either side of any
bridge splitting.  Equation~\eqref{eq:suzuki-colorings} therefore
forces both component counts to be at least $n+1$.  Since
$\chi(G_n)=-2n$, the Euler-characteristic formula gives
\[
 b(G_n)\geq2(n+1)+2n=4n+2.
\]
Together with \eqref{eq:unconstrained-upper}, this proves
$b(G_n)=4n+2$.

Now consider an arbitrary constrained bridge splitting and let $T_+$
be the tangle containing every vertex.  If $T_+$ has $s$ components,
Lemma~\ref{lem:trivial-coloring} and
\eqref{eq:suzuki-colorings} imply
\[
                         3^{n+1}\leq3^s,
\]
so $s\geq n+1$.  Applying Lemma~\ref{lem:weight} to the $2n$
four-valent vertices gives
\[
 |T_+\cap\ell|
    =2s+\sum_{v\in V(G_n)}(\deg(v)-2)
    =2s+4n
    \geq6n+2.
\]
It follows that $b_c(G_n)\geq6n+2$.  For the opposite inequality, start with the eight-intersection
constrained presentation of one summand in
Figure~\ref{fig:constrained-sum}.  Use the same connected-sum
gluing at bridge-sphere intersections, with all vertex-containing
sides matched.  The resulting splitting is constrained and has
$8n-2(n-1)=6n+2$ intersections.  Hence
\begin{equation}\label{eq:constrained-value}
                             b_c(G_n)=6n+2.
\end{equation}
Combining \eqref{eq:unconstrained-upper} and
\eqref{eq:constrained-value} yields
\[
                         b_c(G_n)-b(G_n)=2n.
\]

It remains to check the ravel condition.  Suzuki's $\theta_4$-curve
is almost unknotted, so each constituent cycle is an unknot.  At
each edge connected sum there is a decomposing sphere meeting the
graph in two points.  A circle subgraph either lies on one side
or meets this sphere in both points; in the latter case it is a
connected sum of constituent circles from the two summands.
Those circles are unknots.  For a collection of vertex-disjoint
cycles, at most one cycle can meet the decomposing sphere.  Capping
on either side gives unlink subgraphs in the summands, and gluing
these unlinks along the selected components again gives an unlink.
Induction proves that every circle or link subgraph of $G_n$ is
trivial.  Finally, \eqref{eq:suzuki-colorings} exceeds three, the
number of equal-vertex $R_3$ colorings of a connected planar graph.
Thus $G_n$ is nonplanar and is a ravel.
\end{proof}

\begin{figure}[htbp]
  \centering
  \includegraphics[width=.9\linewidth]{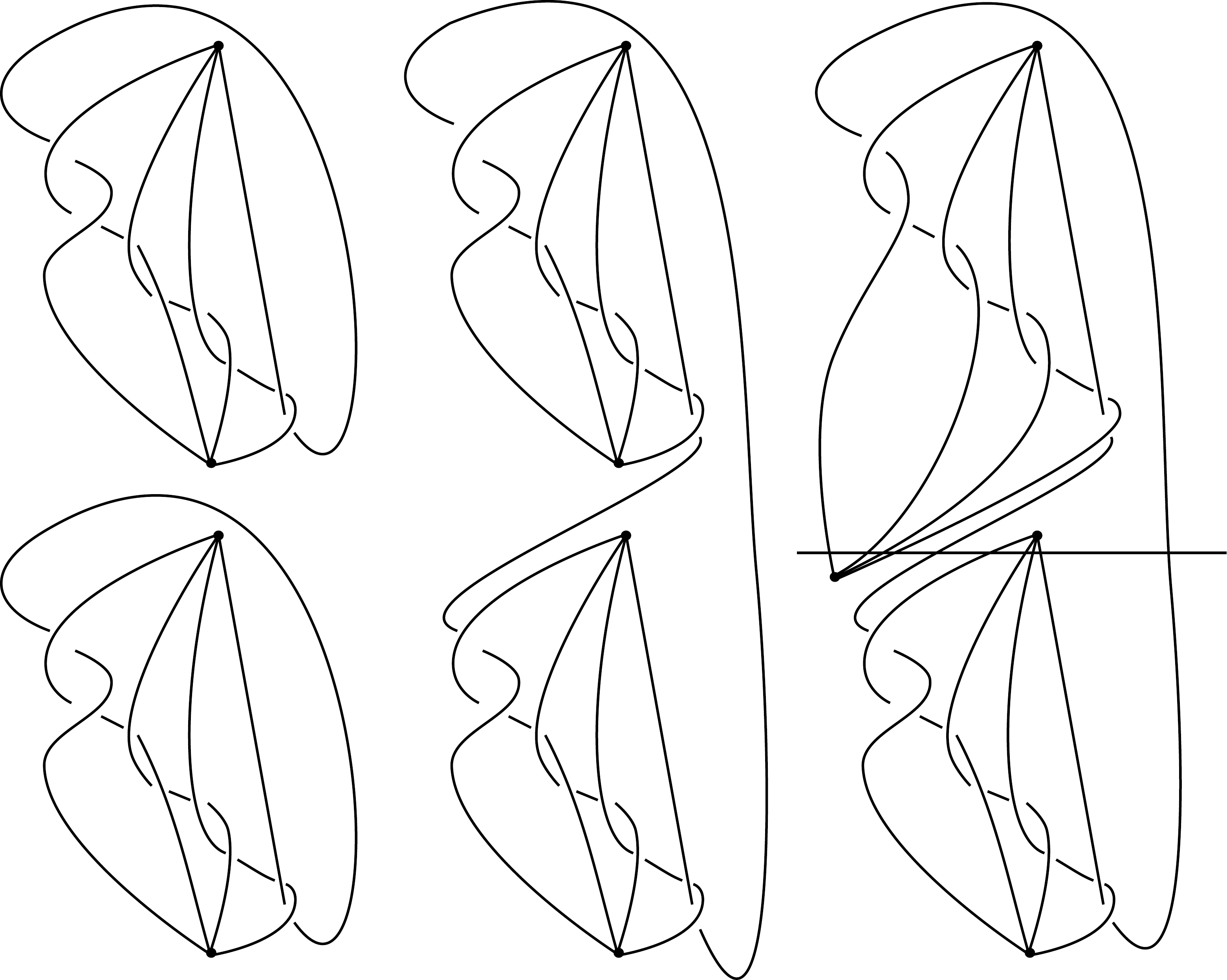}
  \caption{The unconstrained presentations of the connected sums.}
  \label{fig:unconstrained-sum}
\end{figure}

\begin{figure}[htbp]
  \centering
  \includegraphics[width=.85\linewidth]{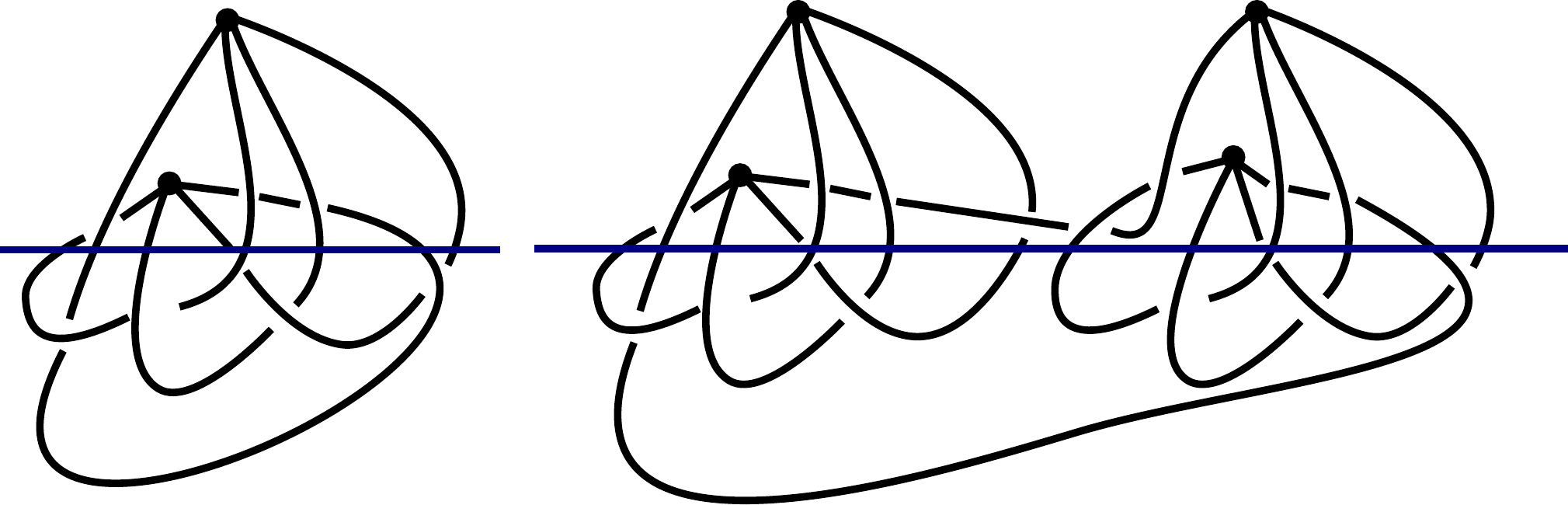}
  \caption{The constrained presentations of the same connected sums.}
  \label{fig:constrained-sum}
\end{figure}

The underlying abstract graph changes with $n$.  This is unavoidable
for an unbounded-gap result.

\begin{prop}\label{prop:fixed-graph}
For every finite abstract graph $\Gamma$ there is a constant
$C(\Gamma)$ such that every virtual embedding $G$ of $\Gamma$
satisfies
\[
                         b_c(G)-b(G)\leq C(\Gamma).
\]
\end{prop}

\begin{proof}
Put a minimizing bridge splitting into the form of two
crossing-free forests separated by a braid, and choose the side
that will contain all vertices.  In a tree on the other side,
choose a vertex $v$ adjacent to a boundary leaf.  Slide $v$ along
that leaf edge and through the braid to the chosen side.  The
other incident edge germs follow in a narrow collar of the same
strand.  Crossings met during this passage are resolved by the
vertex Reidemeister moves; the resulting crossings remain in the
middle braid.  Figure~\ref{fig:bounded} shows the forest part
of this operation.

If $v$ has valence $d$, the old boundary leaf is replaced by
$d-1$ boundary leaves.  On the departure side the tree splits
into trees or arcs; on the arrival side a pod is attached to one
boundary leaf of an existing tree.  Both forests therefore remain
trivial.  The operation adds $d-2$ intersection points, in
particular at most $2d$.  Repeat with a boundary-adjacent vertex
of the remaining forest until no vertices remain on that side.
The result is a constrained splitting.  Summing the coarse bound
$2d$ over the moved vertices gives, for example,
\[
                         C(\Gamma)=2\sum_{v\in V(\Gamma)}\deg(v).
\]
This depends only on the fixed abstract graph.
\end{proof}

\begin{figure}[htbp]
  \centering
  \includegraphics[width=.5\linewidth]{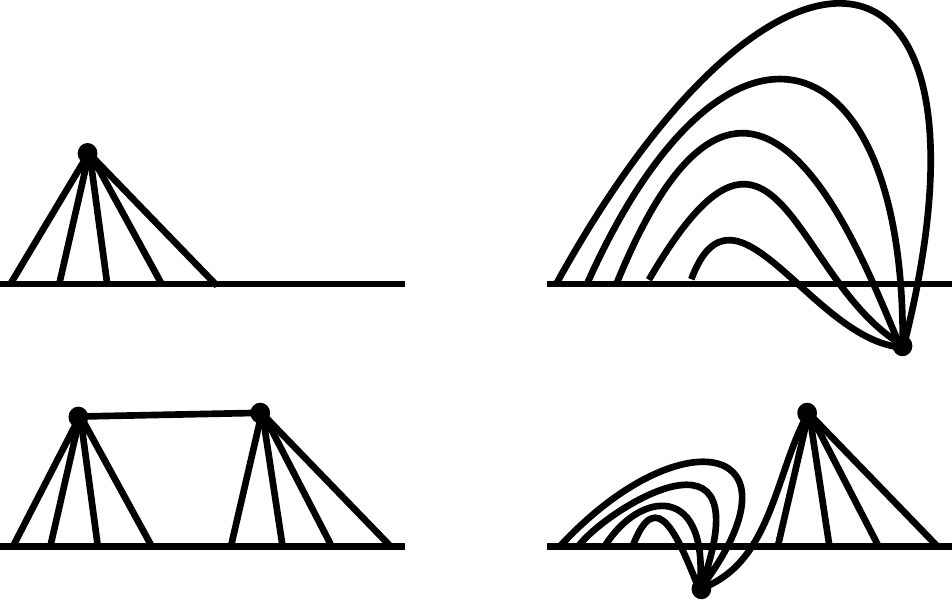}
  \caption{Moving a vertex across the bridge line at bounded cost.}
  \label{fig:bounded}
\end{figure}

\begin{cor}[Constrained graph complexity versus spine complexity]
\label{cor:spine-gap}
The family of Theorem~\ref{thm:main-gap} satisfies
\[
 b_c(G_n)-b_{\mathrm{hb}}(N(G_n))\geq2n=-\chi(G_n).
\]
\end{cor}

\begin{proof}[Proof of Corollary~\ref{cor:spine-gap}]
Corollary~\ref{cor:handlebody-comparison} gives
$b_{\mathrm{hb}}(N(G_n))\leq b(G_n)$.  Hence
\[
 b_c(G_n)-b_{\mathrm{hb}}(N(G_n))
 \geq b_c(G_n)-b(G_n)\geq2n.
\]
An edge connected sum adds Euler characteristics, and the Suzuki
$\theta_4$-curve has Euler characteristic $-2$.  Thus
$\chi(G_n)=-2n$, proving the asserted linear estimate.
\end{proof}

\clearpage
\bibliographystyle{amsplain}
\bibliography{ref}

\end{document}